\documentclass{imsart}
\usepackage{amsthm, amsmath,tikz, amsfonts}
\usepackage[margin=2cm]{geometry}
\usepackage{graphicx}
\usepackage{url}

\usetikzlibrary{patterns}
\numberwithin{equation}{section}
\theoremstyle{plain}
\newtheorem*{theorem*}{Theorem}
\newtheorem*{proposition*}{Proposition}
\newtheorem{theorem}{Theorem}[section]
\newtheorem{lemma}[theorem]{Lemma}
\newtheorem{cor}[theorem]{Corollary}

\theoremstyle{definition}
\newtheorem{definition}[theorem]{Definition}

\newcommand{\R}{\mathbb{R}}

\newcommand{\rank}{\operatorname{rank}}
\newcommand{\Z}{\mathbb{Z}}

\newcommand{\bb}{\mathrm{b}}
\newcommand{\dd}{\mathrm{d}}

\arxiv{math.PR/0000000}

\begin{document}

\begin{frontmatter}

\title{Principal Component Analysis of Persistent Homology Rank Functions with case studies of Spatial Point Patterns, Sphere Packing and Colloids}
\runtitle{PCA of Persistent homology rank functions}

\begin{aug}
  \author{\fnms{Vanessa}  \snm{Robins}\corref{}\thanksref{t2}\ead[label=e1]{vanessa.robins@anu.edu.au}},
  \author{\fnms{Katharine} \snm{Turner}\ead[label=e2]{kate@math.uchicago.edu}}

  \thankstext{t2}{Supported by an ARC Future Fellowship FT140100604 and Discovery Project DP110102888.}

  \runauthor{Vanessa Robins and Katharine Turner}

  \affiliation{}

  \address{The Australian National University; The University of Chicago.\\ 
          \printead{e1,e2}}
  
\end{aug}

\begin{abstract}
Persistent homology, while ostensibly measuring changes in topology, captures multiscale geometrical information.  It is a natural tool for the analysis of point patterns. In this paper we explore the statistical power of the (persistent homology) rank functions. For a point pattern $X$ we construct a filtration of spaces by taking the union of balls of radius $a$ centered on points in $X$, $X_a =  \cup_{x\in X}B(x,a)$.  The rank function $\beta_k(X):\{(a,b)\in \R^2: a\leq b\} \to \R$ is then defined by $\beta_k(X)(a,b) = \rank \left( \iota_*:H_k(X_a) \to H_k(X_b)\right)$ where $\iota_*$ is the induced map on homology from the inclusion map on spaces.  We consider the rank functions as lying in a Hilbert space and show that under reasonable conditions the rank functions from multiple simulations or experiments will lie in an affine subspace.  This enables us to perform functional principal component analysis which we apply to experimental data from colloids at different effective temperatures and of sphere packings with different volume fractions.  We also investigate the potential of rank functions in providing a test of complete spatial randomness of 2D point patterns using the distances to an empirically computed mean rank function of binomial point patterns in the unit square.
\end{abstract}

%
\begin{keyword}
\kwd{topological data analysis}
\kwd{persistent homology}
\kwd{Betti numbers}
\kwd{spatial point patterns}
\kwd{sphere packing}
\kwd{functional principal component analysis}
\end{keyword}

\end{frontmatter}

\section*{Introduction}

Random point patterns arise in a wide variety of application areas from astrophysics to materials science to ecology to protein interactions. 
The points might represent galaxies, colloidal particles, locations of trees, or molecules, and their distribution in space is  indicative of the underlying processes that created the pattern.  
When studying such systems a number of questions arise. 
For example, could a given pattern be generated from a purely random process with no underlying interactions between the objects (points)?  
If this null hypothesis can be rejected then we would like to say whether a proposed theoretical model generates patterns that are consistent with experimental observations.  
We might also need to compare a large number of point patterns and classify them into different groups, or quantitatively track changes over time to understand the dynamics of a system.  

Stochastic geometry provides various tools to characterise random spatial patterns and these have mostly focussed on first and second-order techniques analogous to means and variance in single-variable statistics~\cite{stoyan_stochastic_1995}. 
There are a number of situations where more sensitive tests of structural difference are required~\cite{mecke_morphological_2005}.
Persistent homology is an algebraic topological tool developed for data analysis that is an intuitively appealing measure of higher-order structure and encompasses spatial correlations of all orders~\cite{carlsson_topology_2009,edelsbrunner_computational_2010}.  
This paper shows how to adapt persistent homology information to perform statistical analyses of spatial point patterns, such as principal component analysis (PCA). 

Topology is the study of spatial objects equivalent under continuous deformations --- the old joke is that a topologist can't tell the difference between a coffee mug and a bagel.  The homology groups of a space $X$, $H_k(X)$, $k=0,1,2,\ldots$ are algebraically quantified topological invariants that provide information about equivalent points, loops, and higher dimensional analogs of loops.
Homology detects a $k$-dimensional hole as a $k$-dimensional loop (cycle) that does not bound a $(k+1)$-dimensional piece of the object.     
The ranks of the homology groups are called \emph{Betti numbers}:  $\beta_0$ counts the number of connected components, $\beta_1$ the number of independent non-bounding loops, $\beta_2$ the number of enclosed spaces in a three-dimensional object.   
A solid bagel and a coffee mug both have $\beta_0 = 1$, $\beta_1 = 1$, $\beta_2 = 0$, while their surfaces have $\beta_0 = 1$, $\beta_1 = 2$, $\beta_2 = 1$.  
Homology groups and their Betti numbers are inherently global properties of an object that are sensitive to some geometric perturbations (tearing and gluing) and not to others (continuous deformation).  

Another topological invariant that has recently been used to summarise structure in point patterns is the Euler characteristic signature function~\cite{mecke_morphological_2005,parker_automatic_2013}.
The Euler characteristic is the alternating sum of the Betti numbers (for three-dimensional objects, $\chi = \beta_0 - \beta_1 + \beta_2$).  
It is a topological invariant but also has measure-theoretic properties that make it more amenable to statistical analysis than the Betti numbers, but it is a less sensitive topological invariant by definition. 
Nevertheless, the above studies have demonstrated that the Euler characteristic signature function is sensitive to short-range, higher-order correlations in point patterns.

To determine the homology groups, the topological space must be represented by simple building blocks of points, line segments, surface patches, and so on, that are joined together in a specific way, i.e., as a cell complex. 
The homology groups are then defined via a boundary operator, $\partial_k$, that maps each cell of dimension $k$ onto the cells of dimension $k-1$ in its boundary:
\begin{equation}
	\partial_k :  C_{k} \to C_{k-1} 
\end{equation}
The kernel of $\partial_k$ is called the cycle group $Z_k$ and the image of $\partial_{k+1}$ is called the boundary group $B_k$. 
All boundaries are cycles, so we can form the homology group as the quotient $H_k = Z_k / B_k$.   
See the text by Hatcher for a comprehensive treatment of homology theory~\cite{hatcher_algebraic_2002}, or~\cite{robins_algebraic_2014} for a concise overview aimed at physicists. 

When examining point data, a parameter must be introduced to define which points are connected to one another and so build the cell complex. 
A crucial lesson learnt in the early days of topological data analysis is that instead of trying to find a single best value for this parameter, much more is learned by looking at how the homology evolves over a sequence of parameter values. 
Rather than working with single cell complex we work with a \emph{filtration}, a family of spaces $K_a$ such that $K_a\subset K_b$ whenever  $a\leq b$. 
Often this parameter $a$ is a length scale, so that although we are measuring topological quantities, the way these change tells us about the geometrical features of the data set.  
For example, if we used a filtration of $\mathbb{R}^3$ defined by lower level sets of the distance function to the surface of an ideally-smooth bagel, we would be able to read the radius of the hole of the bagel from the function $\beta_1(K_a)$ as it is at that radius that the space of loops changes from one to zero dimensional. 
If the bagel is a real one with irregular cross-section, bumps and splits, the Betti number function $\beta_1(K_a)$ will not be a  clean step function as it is in the ideal case, but may jump around and obscure the exact point of the large-scale change from bagel to blob. 

The issue of topological noise, i.e., the lack of stability of the Betti numbers, motivated the development of persistent homology in the 1990s~\cite{verri_use_1993,robins_towards_1999,edelsbrunner_topological_2002}.    
The inclusion of $K_a \subset K_b$ for $a < b$ induces a homomorphism between the homology groups $H_k(K_a)$ and $H_k(K_b)$ that tells us which topological features persist from $K_a$ to $K_b$ and which disappear (i.e., get filled in).  
The persistent homology group is the image of $H_k(K_a)$ in $H_k(K_b)$, it encodes the $k$-cycles in $K_a$ that are independent with respect to boundaries in $K_b$: 
\begin{equation}
    H_k(a,b) := Z_k(K_a)/ (B_k(K_b) \cap Z_k(K_a)).
\end{equation}
Algorithms for computing persistent homology from a given filtration are quite simple in their most basic form~\cite{edelsbrunner_topological_2002,zomorodian_computational_2009}, and are now implemented efficiently in a number of freely-downloadable packages~\cite{morozov_dionysus_2013,tausz_javaplex_2014,bauer_phat_2014,nanda_perseus_2015}. 

The two most common ways of representing persistent homology information are the barcode~\cite{carlsson_persistence_2005} and the persistence diagram~\cite{cohen-steiner_stability_2007}.  
The barcode is a collection of intervals $\left[ b,d \right)$ each representing the birth, $b$, and death, $d$, values of a persistent homology class. 
Equivalently, the persistence diagram is a set of points $(b,d)$ in the plane.  
The main problem with these objects is that they are very difficult to work with statistically.  
Distances between persistence diagrams and definitions of their means  and variance require advanced analytical techniques~\cite{bubenik_statistical_2010,turner_frechet_2012,munch_probabilistic_2015}.   

In this paper we return to the above definition of persistent homology groups and quantify them via their rank, 
\begin{equation}
	\beta_k (a,b) := \text{rank} H_k(a,b), \text{ for } a< b . 
\end{equation}
This \emph{persistent homology rank function} is an integer-valued function of two real variables and can be thought of as a cumulative distribution function of the persistence diagram. 
Since the persistent homology rank function is just a function we can apply standard statistical techniques to analyse distributions of them.  
This rank function is related to the size function~\cite{verri_use_1993} and has also been defined for multidimensional persistence, in the case of filtrations that are built using two or more parameters~\cite{carlsson_theory_2009,cerri_betti_2013}. 
Other functional summaries of persistence diagrams have also been proposed  recently and termed \emph{persistence landscapes and silhouettes}, see~\cite{bubenik_statistical_2015,chazal_stochastic_2014}.  
The persistence landscapes $\lambda_k(i,t)$ are a family of functions so that for each $i = 1,2,3,\ldots$, $\lambda_k(i,t)$ is a function of a single variable (the subscript $k$ is the homology dimension as above).   
The variable $t$ is related to the persistence diagram coordinates by $t = (b+d)/2$.  
We argue that the rank functions used here, $\beta_k(a,b)$, retain a more direct connection to the geometry and topology of the original data, although they are functions of two variables. In particular they are more suitable for analyzing distributions of local point configurations.
 
This functional approach to persistent homology (either using landscapes or rank functions) greatly simplifies the business of ``doing statistics'' with persistent homology.  
It provides a framework where it is simple to compute averages and variances of these topological signatures from many data sets generated by a particular system, and to make statistically rigorous statements about whether an experimentally-observed pattern is compatible with a theoretical model.
In particular the pointwise average of $\beta_k(a,b)$ is a function on $\R^{2+}$ which tells us the expected ranks of the corresponding persistent homology groups.  
We can also perform functional principal component analysis (PCA) where the principal component functions are also functions on $\R^{2+}$ containing topological information.  

In section \ref{sec:background} we provide the necessary background to define the persistent homology rank function (from now on shortened to rank function).  We then introduce a Hilbert space of functions in which the rank functions exist in section~\ref{sec:metric}. 
In section \ref{sec:PCA} we review the definition of functional PCA and explain how to compute the principal component functions and scores using the dot product matrix derived from a set of functions.

The applications we consider in this paper all pertain to point processes in some way. In section \ref{sec:CSR}, we propose a test for whether point processes are completely spatially random using the distribution of distances to the mean persistent homology rank function and compare its power to other standard tests for a variety of models of point processes. The final two sections apply functional PCA to real world data sets.  We first analyze the distributions of colloidal particles at a fluid interface from an experiment designed to test theories of melting of two-dimensional crystals.  The first principal component explains much of the variation ($>96\%$ for dimension 0 and $>86\%$ in dimension 1).  Furthermore, it separates the colloids into the separate phase  groups.
We then study random sphere packings using the locations of centers of the spheres derived from x-ray images and show that in both dimension 1 and dimension 2 persistent homology rank functions, the first principal component explains $>97\%$ of the variation and is highly related to volume fraction of the packings. 
Persistent homology gives topological information about the distributions of local arrangements for more and less efficient packings.

\section{Persistent homology rank functions}\label{sec:background}

To define persistent homology rank functions we will need to first provide some background theory on homology and persistent homology groups. 
The simplest homology theory to define is simplicial homology, whose ingredients are a topological space defined by a simplicial complex and a coefficient group, which for persistent homology computations is usually the field $\Z_2$.  


A \emph{$k$-simplex} is the convex hull of $k+1$ affinely independent points $v_0,v_1, \ldots v_k$ and is denoted $[v_0,v_1,\ldots,v_k]$. For example, the $0$-simplex $[v_0]$ is the vertex $v_0$, the $1$-simplex $[v_0,v_1]$ is the edge between the vertices $v_0$ and $v_1$ and the $2$ simplex $[v_0, v_1, v_2]$ is the triangle bordered by the edges $[v_0,v_1]$, $[v_1, v_2]$ and $[v_0, v_2]$. Technically, there is an orientation on simplices. If $\tau$ is a permutation then $[v_0,v_1,\ldots,v_k] = (-1)^{\operatorname{sgn}(\tau)}[v_{\tau(0)}, v_{\tau(1)}, \ldots , v_{\tau(k)}]$. However, if we are considering homology over $\Z_2$ then $1=-1$ and we can ignore orientation.

We call $[u_0, u_1, \ldots u_j]$ a \emph{face} of $[v_0,v_1, \ldots v_k]$ if $\{u_0, u_1, \ldots u_j\}\subset \{v_0,v_1, \ldots v_k\}$. A \emph{simplicial complex} $K$ is a countable set of simplices such that
\begin{itemize}
\item Every face of a simplex in $K$ is also in $K$.
\item If two simplices $\sigma_1,\sigma_2$ are in $K$ then their intersection is either empty or a face of both $\sigma_1$ and $\sigma_2$.
\end{itemize}

Given a finite simplicial complex $K$, a \emph{simplicial $k$-chain} is a formal linear combination (with coefficients in our field of choice) of $k$-simplices in $K$. The set of $k$-chains forms a vector space $C_k(K)$. We define the boundary map $\partial_k:C_k(K) \to C_{k-1}(K)$ by setting
\begin{displaymath}
\partial_k([v_0, v_1, \ldots v_k]) = \sum_{j=0}^k (-1)^j[v_0,\ldots \hat{v_j}, \ldots v_k]= \sum_{j=0}^k [v_0,\ldots \hat{v_j}, \ldots v_k]
\end{displaymath}
for each $k$-simplex and extending to $k$-chains using linearly. The second equality follows because our coefficient group is $\Z_2$ where $-1=1$.

Elements of $B_k(K) = \operatorname{im} \partial_{k+1}$ are called boundaries and elements of $Z_k(K)=\operatorname{ker} \partial_k$ are called cycles. Direct computation shows $\partial_{k+1}\circ\partial_k=0$ and hence $B_k(K) \subseteq Z_k(K)$. This means we can define the $k^{th}$ \emph{homology group} of $K$ to be
\begin{displaymath}
H_k(K):=Z_k(K)/B_k(K).
\end{displaymath}
A comprehensive introduction to homology can be found in \cite{hatcher_algebraic_2002}.

We now consider the definition of persistent homology groups of a filtration.  
A filtration $K = \{K_r | r\in \R\}$ is a countable simplicial complex indexed over the real numbers such that each $K_a$ is a simplicial complex and $K_a \subseteq K_b$ for $a \leq b$. 
We wish to describe how the topology of the filtration changes as the parameter increases. 
For $a\leq b$ we have an inclusion map of simplicial complexes $\iota: K_a \to K_b$ that induces inclusion maps 
\begin{displaymath}
	\iota:B_k(K_a) \to B_k(K_b) \quad \text{and}\quad \iota: Z_k(K_a) \to Z_k(K_b).
\end{displaymath}	
These inclusions induce homomorphisms (which are generally not inclusions) on the homology groups: 
\begin{displaymath}
	\iota_k^{a\to b}:H_k(K_a) \to H_k(K_b).
\end{displaymath} 
The image of $\iota_k^{a\to b}$ consists of equivalence classes of cycles that were present in $K_a$, where the homological equivalence is measured with respect to boundaries in $K_b$.  
We therefore define the \emph{persistent homology group} for $a\leq b$ as
\begin{displaymath}
H_k(a,b)= \iota( Z_k(K_a) ) / (B_k(K_b) \cap \iota( Z_k(K_a)) ).
\end{displaymath}
It is worth observing that $H_k(a,a) = H_k(K_a)$ and hence $\beta_k(a,a)=\beta_k(K_a)$ which motivates the use of the symbol $\beta$ for the rank function as it is a generalization of the traditional Betti numbers.

Let $\R^{2+}:=\{(x,y) \in (-\infty \cup \R) \times (\R\cup \infty): x<y\}$. 
We define the $k$-th dimensional persistent homology rank function corresponding to the filtration $K$ to be
\begin{align*}
\beta_k(K): \R^{2+} & \to \Z\\
(a,b) & \mapsto \dim H_k(a,b)
\end{align*}
There are two alternative arguments for the name ``rank''. Firstly, $\beta_k(K)(a,b)$ is the rank of $H_k(a,b)$ viewed as a module over $\Z_2$. Alternatively, since $H_k(a,b)$ is isomorphic to $\iota_k^{a\to b}(H_k(K_a))$ we know that $\beta_k(K)(a,b)$ is the rank of the function $\iota_k^{a \to b}:H_k(K_a) \to H_k(K_b)$.

We call a filtration \emph{tame} if $ \dim H_k(a,b)$ is finite for all $a < b$. From now on we will restrict our attention to tame filtrations (and hence everywhere finite valued rank functions). This restriction will be satisfied by any application involving finite  data.

The rank function contains the same information as the persistence diagrams and barcodes used frequently in computational topology literature. 
The persistence diagram in dimension $k$, denoted PD$k$,  consists of points $(\bb(\alpha),  \dd(\alpha) ) \in \R^{2+}$ where $\bb(\alpha)$ is the \emph{birth} value of a homology class and $\dd(\alpha)$ is the \emph{death} value of the filtration parameter. 
We say that a homology class $\alpha \in H_k(K_{\bb(\alpha)} )$ is \emph{born} at  $\bb(\alpha)$ if it is in the cokernel of $\iota_k^{s \to \bb(\alpha) }$ for any $s<\bb(\alpha)$.  
That is, $\alpha$, is not the image of any cycle that occurs earlier in the filtration. 
The homology class of $\alpha$ \emph{dies} at  $\dd(\alpha)$ if for $\bb(\alpha) < t <  \dd(\alpha)$ we have 
$\alpha \notin \ker  \iota_k^{\bb(\alpha) \to t}$ but $\alpha \in \ker  \iota_k^{\bb(\alpha) \to \dd(\alpha)}$. 
Informally we can think of the process of dying as the cycle becoming a boundary or two cycles merging.  
The difference $\dd(\alpha)- \bb(\alpha)$ is termed the \emph{persistence} of the cycle $\alpha$.  
Some cycles can have $\dd(\alpha) = \infty$, these are called \emph{essential classes}.  

Each rank function can be thought of as a cumulative function of a persistence diagram. This is because as the quantity $\beta_k(a,b)$ is the same as the number of points in the persistence diagram in the region $(-\infty, a] \times [b,\infty)$. The correspondence between persistence diagrams and rank functions makes it easy to compute the persistence diagrams from rank functions and vice versa. This is convenient as there are many available computer packages to compute persistence diagrams~\cite{morozov_dionysus_2013,tausz_javaplex_2014,bauer_phat_2014,nanda_perseus_2015}.

Rank functions have some nice monotonicity properties~\cite{cohen-steiner_stability_2007,carlsson_theory_2009}. 
If $\beta$ is a persistent homology rank function then for $a\leq c \leq b\leq d$
\begin{equation}\label{eq:mono} 
\beta(c,b)- \beta(a,b)-\beta(c,d)+\beta(a,d)\geq 0.
\end{equation}
This is because the inclusion exclusion principle tells us that $\beta(c,b)- \beta(a,b)-\beta(c,d)+\beta(a,d)$ should be the number of points in the corresponding persistence diagram lying in the region $(a,c]\times[b,d)$. This is depicted in Figure \ref{fig:mono}

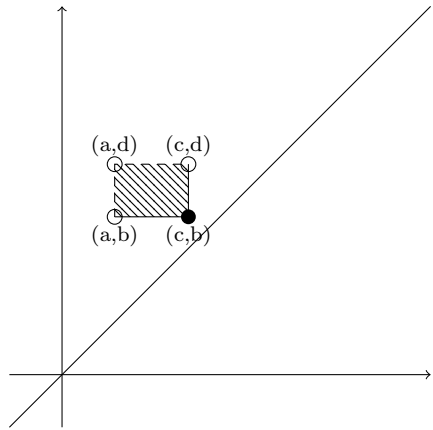
\begin{figure}[hbt]
\begin{center}
\begin{tikzpicture}[scale=.7]
\draw[->] (0,-1)--(0,7);
\draw[->](-1,0)--(7,0);
\draw[](-1,-1)--(7,7);

\draw (1,4) circle (4pt)  node [above] {(a,d)};
\draw (1,3) circle (4pt)  node [below] {(a,b)};
\fill (2.4,3) circle (4pt)  node [below] {(c,b)};
\draw(2.4,4) circle(4pt)  node [above] {(c,d)};
\fill[pattern=north west lines, pattern color=black] (1,3)--(1,4)--(2.4,4)--(2.4,3)--(1,3);
\draw (1,3)--(2.4,3);
\draw(2.4,3)--(2.4,4);
\draw[dashed](1,3)--(1,4)--(2.4,4);
\end{tikzpicture}
 
\caption{$\beta(c,b)- \beta(a,b)-\beta(c,d)+\beta(a,d)$ for the rank function $\beta$ is the number of points in the corresponding persistence diagram lying in the shaded region}
\label{fig:mono}
\end{center}
\end{figure}


 
A common way to produce filtrations is to consider sublevel or superlevel sets of a continuous function. Given $f:X \to \R$ we can set $K_t=f^{-1}(-\infty, t]$ for a filtration by sublevel sets. The filtration by superlevel sets of $f$ is effectively the same as the filtration by sublevel sets of $-f$. 
Combinatorial cell complexes derived from scalar functions with non-degenerate critical points are called \emph{Morse complexes}.  
An important example that we will focus on in this paper is the distance function to a set of points. 
The filtration induced by the distance to a set of points is modelled by a simplicial complex called the \v Cech complex.  
Let $S$ be the set of points $\{v_1, \ldots v_N \} \subset  \R^d$. The \v Cech filtration, $\{\mathcal{C}(S,r): r\geq 0\}$,  is the filtration over the complete $N$-simplex where
\begin{displaymath}
[v_{i_1},\ldots ,v_{i_k}] \in \mathcal{C}(S)_r \text{ whenever }\cap_{j=1}^k B(v_{i_j}, r) \neq \emptyset.
 \end{displaymath}
The Nerve Lemma \cite{edelsbrunner_union_1995} states that $\mathcal{C}(S, r)$ is homotopic to $\cup_{v \in S} B(v, r)$. 
This means the \v Cech filtration has the same persistent homology as the filtration by sub level sets of the distance function.
Computationally, the \v Cech filtration can be efficiently implemented for points in $\R^2$ and $\R^3$ by alpha shape subcomplexes of the Delaunay triangulation~\cite{edelsbrunner_three-dimensional_1994}.
 
Another common alternative filtration for point data used in Topological Data Analysis is the Rips filtration. 
The Rips filtration, $\{\mathcal{R}(S,r): r\geq 0\}$,  is the filtration of the complete $N$-simplex where
 \begin{displaymath}
 [v_{i_1},\ldots ,v_{i_k}] \in \mathcal{C}(S)_r \text{ whenever } d(v_{i_j},v_{i_l})\leq r \text{ for all } j,l.
 \end{displaymath}
It is the flag complex whose $1$-skeleta agree with that of the \v Cech complex. The Rips complex is often used in high-dimensional computations as it only requires knowledge of pairwise distances. 
Unfortunately it loses subtle local geometric structure; for example, it cannot distinguish triangles and tetrahedra whose vertices are nearest-neighbours. 

\section{Rank functions are a subset of a Hilbert space of functions} \label{sec:metric}

Persistent homology rank functions lie in the space of real valued functions from the Riemannian metric space $(\R^{2+},g)$ to $\R$. 
We consider a metric $d(\cdot,\cdot)$ on this space of functions defined by 
\begin{align}\label{metric}
d(f,h)^2  = \int_{\R^{2+}}(f-h)^2\, d\mu
\end{align}
where $\mu$ is the measure on $\R^{2+}$ corresponding to the metric $g$ on $R^{2+}$. 
There are many choices of the metric $g$ and hence also its corresponding measure $\mu$. 
A potential problem is that depending on the choice of $\mu$, the pairwise distances between rank functions may be infinite. 
For example, if we choose $\mu$ to be the standard metric for the plane restricted to $\R^{2+}$, (i.e., Lesbesgue measure $\lambda$) the distance between two rank functions can only be finite when their respective sets of essential cycles have identical birth times. 
Otherwise, there is an infinite region in the plane where the rank functions differ and the Lesbesgue measure of this region is infinite. 
To address this issue, we consider different measures on $\R^{2+}$ obtained by multiplying Lesbesgue measure by a weighting function,  $\mu((x,y)) = \phi(y-x)\lambda((x,y))$. 
Making $\phi$ a function of $(y-x)$ means the measure $\mu$ is a function of the persistence or lifetime of each homology class.  
Given a weighting function $\phi$ the distance function is then 
\begin{align}\label{metric}
d_\phi(f,h)^2  = \int_{x<y}(f-h)^2 \phi(y-x)\, dx\, dy.
\end{align}
The choice of weighting function is analogous to the choices made in multivariate PCA where rescaling the original coordinates (by changing the units of measure) will affect the PCA components.

The following Lemma establishes sufficient conditions to ensure rank functions are a finite distance apart.  
Then given a set of rank functions, each pairwise a finite distance apart, we can consider the affine space that they lie in and so construct means and principal components.   

\begin{lemma}\label{lem:finitedist}
Let $\phi:[0,\infty) \to [0,\infty)$ such that $\int_0^\infty \phi(t) \,dt <\infty$.
Let $f_K$ and $f_L$ be rank functions constructed from filtrations $K_t$ and $L_t$. 
If $K_\infty$ and $L_\infty$ are finite simplicial complexes with $H_*(K_{-\infty})=H_*(L_{-\infty})$ and $H_*(K_{\infty})=H_*(L_{\infty})$ then $d_\phi(f_K, f_L)<\infty$. 
\end{lemma}

\begin{proof}
Let $X_K$ and $X_L$ be the persistence diagrams corresponding to $f_K$ and $f_L$.
Since $K_\infty$ and $L_\infty$ are finite simplicial complexes there exists finite $m$ and $M$ such that every point in the persistence diagrams is of the form $(-\infty, b), (-\infty,\infty), (a,b)$ or $(a,\infty)$ where $m\leq a,b\leq M$. There also exists an $N$ such that $0\leq f_K(x,y), f_L(x,y) \leq N$ for all $(x,y)\in \R^{2+}$.

The assumption that $H_*(K_{-\infty})=H_*(L_{-\infty})$ implies that $X_K$ and $X_L$ have a number of points whose first coordinate is $-\infty$. This implies that the $f_K(x,y)=f_L(x,y)$ whenever $y<m$. Similarly $f_K(x,y)=f_L(x,y)$  whenever $x>M$.

\begin{align*}
d(f_K,f_L)^2  &= \int_{x<y}(f_K-f_L)^2 \phi(y-x)\, dx\, dy   \\
& \leq \int_{\{(x,y): x<y, y>m, x<M\}}N^2 \phi(y-x)\, dx\, dy  \\
& < \infty
\end{align*}

\end{proof}

The sufficient conditions expressed in Lemma \ref{lem:finitedist} are by no means exhaustive. 
However, practically speaking, whenever the input is finite (or from a finite discretization) and the starting and ending stages of the filtrations are the same, the rank functions will lie in the same affine space and this will cover almost any real application.
In particular, rank functions computed on the \v{C}ech complexes of point patterns, the colloid and the sphere packing data all satisfy the conditions in Lemma \ref{lem:finitedist}.  

A scenario not covered by  Lemma \ref{lem:finitedist} is if we considered rank functions constructed from the \v Cech complexes on two point patterns but one of the point patterns lived in the unit square and the other lived in the flat torus. In this example the pairwise distance will be infinite.

Now fix a function $h:\R^{2+} \to \R$ and consider the space of functions, denoted $A(h)$, that are a finite distance from $h$. This is an affine space that can be transformed to a vector space, $V(h)$, by centering about $h$:  $V(h)=\{f-h : \: f \in A(h)\}$. 
We can put an inner product structure on $V(h)$ using 
\begin{displaymath}
\langle v_1, v_2 \rangle = \int_{\R^{2+}} v_1 \, v_2 \,d\mu.
\end{displaymath}

We will generally use the mean persistent rank function for centering.  
Let $\beta^1, \beta^2, \ldots \beta^n$ be $k$-dimensional rank functions.   
We can define a mean persistent homology rank function simply as 
\begin{displaymath}
 \bar{\beta}(a,b) = \frac{1}{n}\sum_{i=1}^n \beta^i(a,b)
 \end{displaymath}
for all $(a,b) \in \R^{2+}$. 
This definition retains topological meaning because $\bar{\beta}(a,b)$ is the expected dimension of $H_k(a,b)$.
Similarly, we can define a mean persistent homology rank function over a distribution of rank functions by taking the appropriate pointwise defined integrals.
The mean rank function is unlikely to be a persistent homology function as it will probably have some non-integer values. However it is easy to check that it will still retain the monotonicity properties of persistent homology functions \eqref{eq:mono}.

The following corollary holds in more general contexts but is sufficient for our needs.
\begin{cor}\label{cor:affine}
Let $X$ be a manifold with bounded curvature and finite dimensional homology (such as $X=\R^d$). Let $\{\beta^1, \beta^2, \beta^3, \ldots \beta^n\}$ either be the set of persistent homology rank functions from the \v Cech filtrations of different finite point clouds in $\R^d$ or from the filtrations of sub level sets of kernel density functions from different finite point clouds in $X$. Let $\bar{\beta}$ be the (pointwise) mean of the $\beta^i$. Then the $\{\beta^1, \beta^2, ... , \beta^n\} \subset A(\bar{\beta})$ and $\{\beta^1-\bar{\beta}, \beta^2-\bar{\beta}, ... \beta^n-\bar{\beta}\}$ lie in an at most $(n-1)$-dimensional vector space.  
\end{cor}

\section{Application: Testing whether point patterns are completely spatially random}\label{sec:CSR}

One of the basic questions in the field of spatial point process analysis is testing the hypothesis that a point process comes from some particular model. 
The method we outline here can be used for any model (after fixing all parameters) but we have restricted ourselves to the case of testing whether point patterns are completely spatially random --- a common null hypothesis. 
A spatial point process is completely spatially random (CSR) if there is no interaction between the locations of the points. 
The two main CSR models are (homogeneous) Poisson point processes and (homogeneous) Binomial point processes. 
\begin{definition}
Let $f$ be a probability density function on a set $A$ and let $n \in \mathbb{N}$. A point process $X$ consisting of $n$ independent and identically distributed points with common density $f$ is called a \emph{binomial point process} of $n$ points in $A$ with density $f$. If $f$ is constant then the binomial point process is called homogeneous and is CSR. 
\end{definition}
\begin{definition}
A point process $X$ on $A$ is a \emph{Poisson point process} with intensity function $\rho$ (and intensity measure $\mu$) if the following two properties hold:
\begin{itemize}
\item For all $B \subset A$, the number of points in $X$ lying in $B \subset A$ follows a Poisson distribution with mean $\mu(B)$.
\item For $B, C \subset A$ with empty intersection, the number of points in $X$ lying in $B$ and $C$ are independent.
\end{itemize}
\end{definition}
These two CSR models are intimately related --- a Poisson distribution is a binomial point process $X$ of $P(\lambda)$ points, where $P(\lambda)$ is a drawn from a Poisson distribution.

Here we propose a test for point patterns being CSR using persistent homology rank functions. The persistent rank functions corresponding to the dimensions $0$ and $1$ each provide a different test. The first step is to estimate empirically the mean persistent rank function $\bar{\beta_k}$. Then using a new set of CSR patterns we estimate the distribution of distances squared to $\bar{\beta_k}$. We then determine a cutoff $C_k$ depending on a chosen p-value (i.e., confidence level). 
New point patterns are then deemed not CSR if their distance to $\bar{\beta_k}$ is more that $C_k$.

We used this test with point patterns drawn from a variety of point process models to demonstrate its power.   
The same point patterns were also tested using some standard CSR tests for the sake of comparison and the results are presented in Table \ref{tab:power}. 
The Completely Spatially Random model used for each test is a set of $100$ i.i.d. points patterns in the unit square. 
The estimated mean persistence rank functions $\bar{\beta_0}$, $\bar{\beta_1}$ were computed using $300$ CSR point patterns. 
The distributions of the distances squared to $\bar{\beta_k}$ were computed using an additional independent set of $200$ CSR point patterns.  These provided cutoffs of $C_0=0.05$ and $C_1=0.01$ corresponding to the $p$-value of $0.05$. 

\begin{figure}[h]
(a)\includegraphics[height=2in]{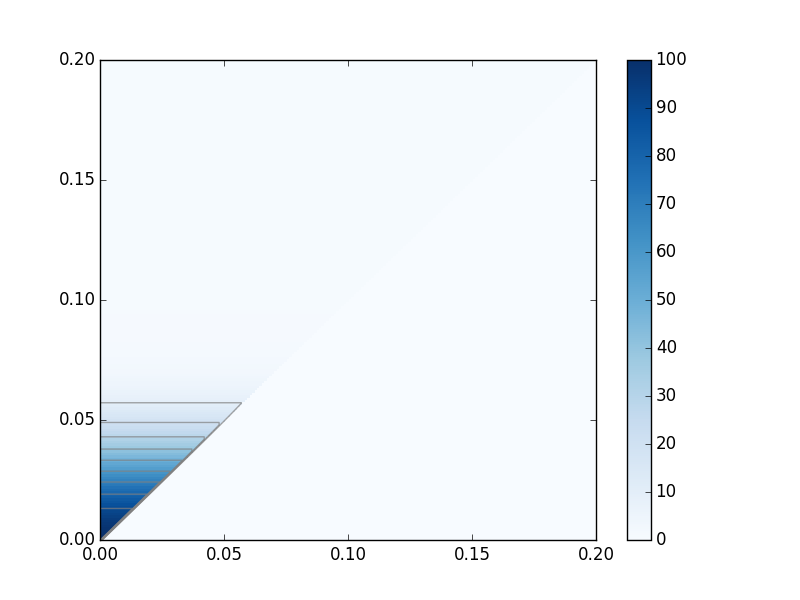}
(b)\includegraphics[height=2in]{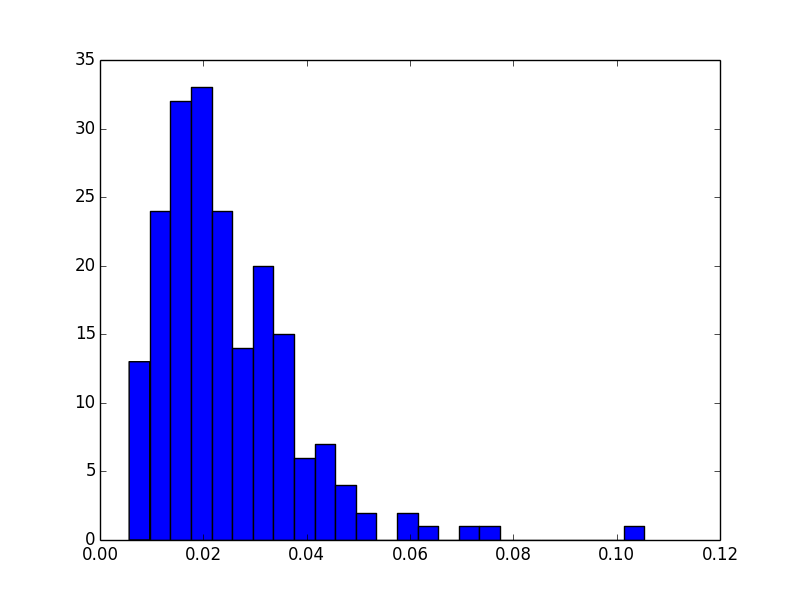}
\caption{(a)The mean $H_0$ persistent rank function of $300$ i.i.d. CSR point patterns of $100$ points in the unit square. (b) The distances squared of $200$ independent CSR patterns from the mean rank function depicted in (a).}  
\end{figure}
\begin{figure}[h]
(a)\includegraphics[height=2in]{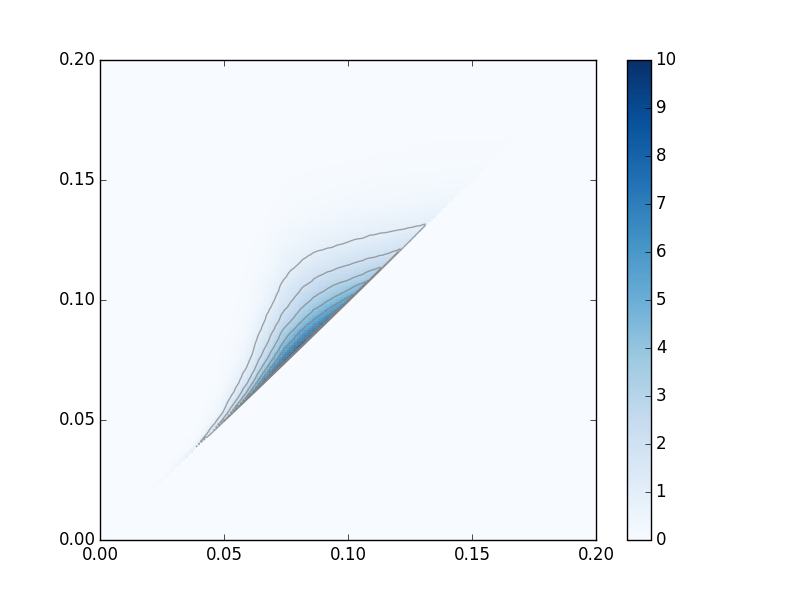}
(b)\includegraphics[height=2in]{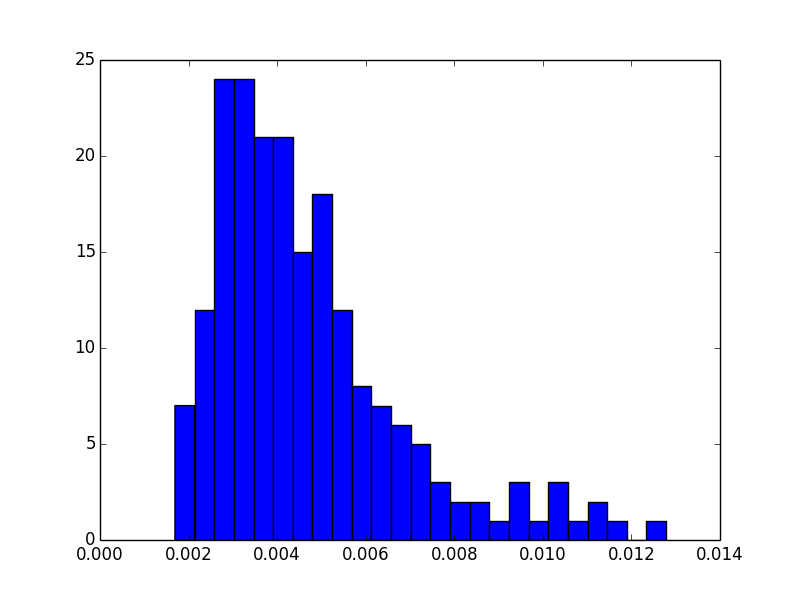}
\caption{(a)The mean $H_1$ persistent rank function of $300$ i.i.d. CSR point patterns of $100$ points in the unit square. (b) The distances squared of $200$ independent CSR patterns from the mean rank function depicted in (a).}  
\end{figure}

We tested $100$ point patterns each from a variety spatial point processes in the unit square, each conditioned on the number of points being exactly $100$ (to match the Binomial CSR model). 
The four models are another set of CSR patterns,  Baddeley-Silvermann, the Strauss pairwise-interaction model and a Matern clustering model \cite{baddeley_spatstat_2005}.
We then compare our results with those given by other standard tests on the same set of simulated point patterns. 
These other tests were the Diggle-Cressie-Loosmore-Ford  test (DCLF), Maximum Absolute Deviation test (mad), quadrant test (quad) and Clark and Evans test (Clark Evans) all available in the `spatstat' R package \cite{baddeley_spatstat_2005}.

\begin{figure}[h]
\begin{center}
\begin{tabular}{c|c|c|c|c}\label{tab:power}
 & CSR 	& 		Strauss & Matern Cluster & Baddeley-Silverman \\
 \hline
DCLF 	& 2		&1		&98			&0\\
Mad 		& 0		&0		&98			&0\\
Quad 	& 3		&10		&97			&1\\
Clark Evans & 5	& 56		&67			&96\\
dim 0 	& 3		&73		&14			&99\\
dim 1 	& 7		&17		&59			&81\\
\end{tabular}
\end{center}
\caption{The number of point patterns (out if $100$ testing) for which we rejected the null hypothesis that they were CSR. Here we compare the results from the persistent homology rank test we constructed with some standard tests available in the $R$-package spatstat \cite{baddeley_spatstat_2005}.} 
\end{figure}

We now will give  a brief overview of theses models and all the parameters we chose.
\begin{figure}[hbt]
(a)\includegraphics[height=2in]{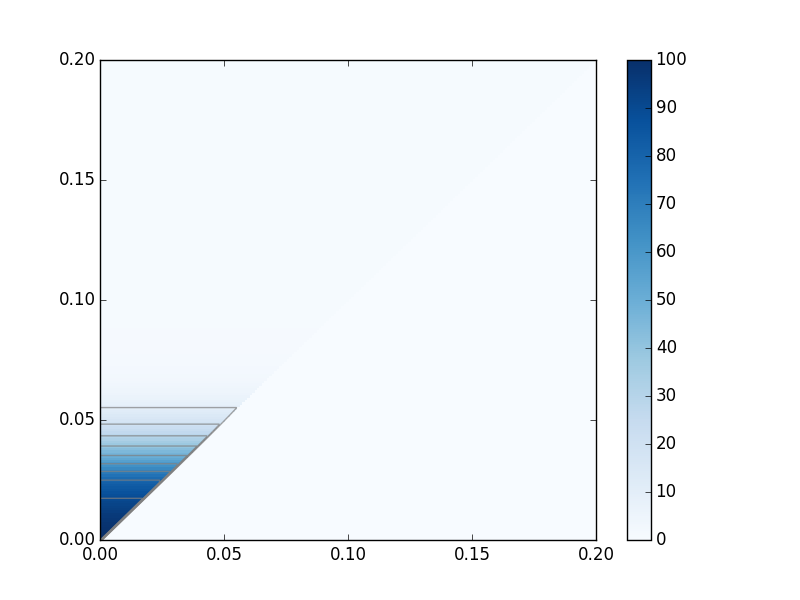}
(b)\includegraphics[height=2in]{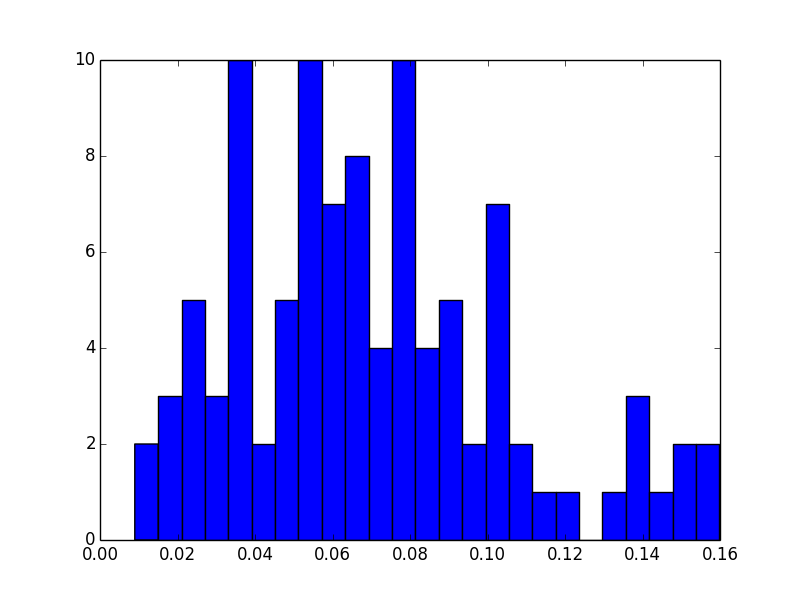}
(c)\includegraphics[height=2in]{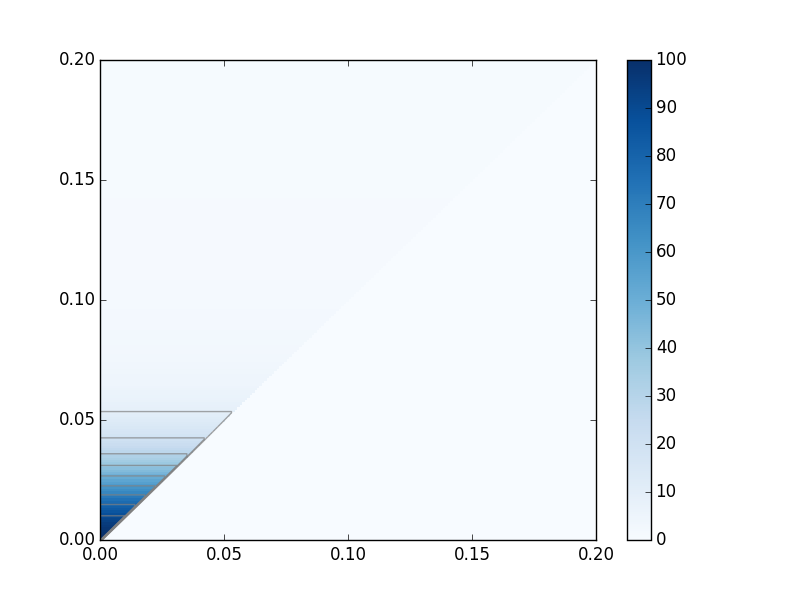}
(d)\includegraphics[height=2in]{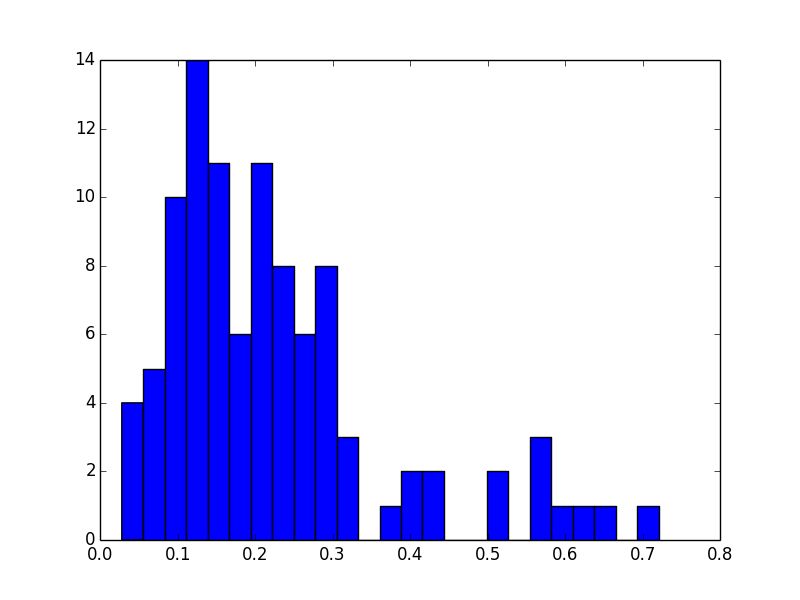}
(e)\includegraphics[height=2in]{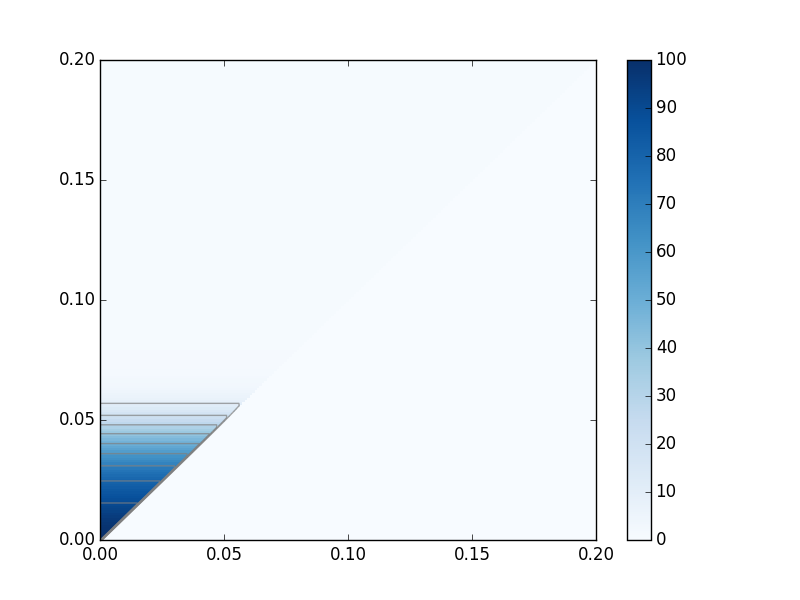}
(f)\includegraphics[height=2in]{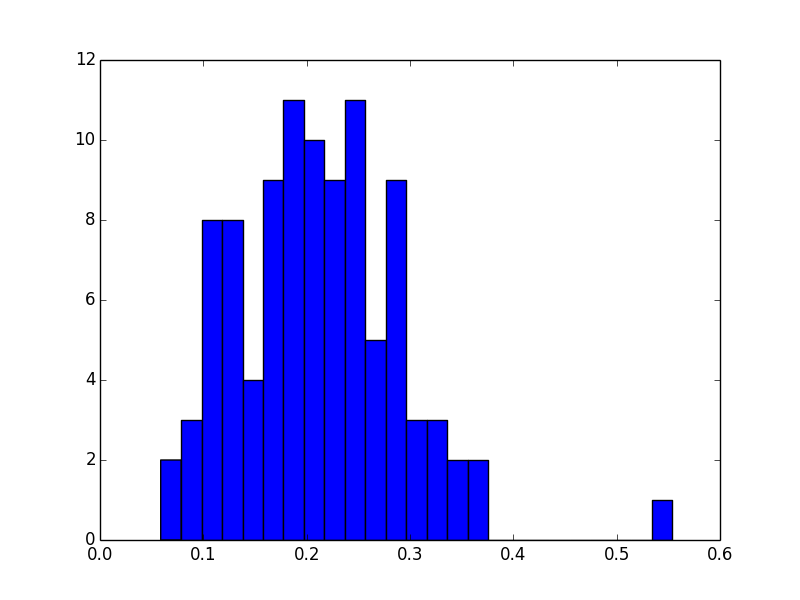}
\caption{(a), (c), (e) are the mean $H_0$ persistent rank function of $100$ i.i.d. spatial point patterns from the models of Strauss, Matern clustering and Baddeley-Silvermann, respectively.
(b), (d) and (f) are the corresponding distributions of distances squared to $\bar{\beta_0}$.}
\end{figure}

\begin{figure}[hbt]
(a)\includegraphics[height=2in]{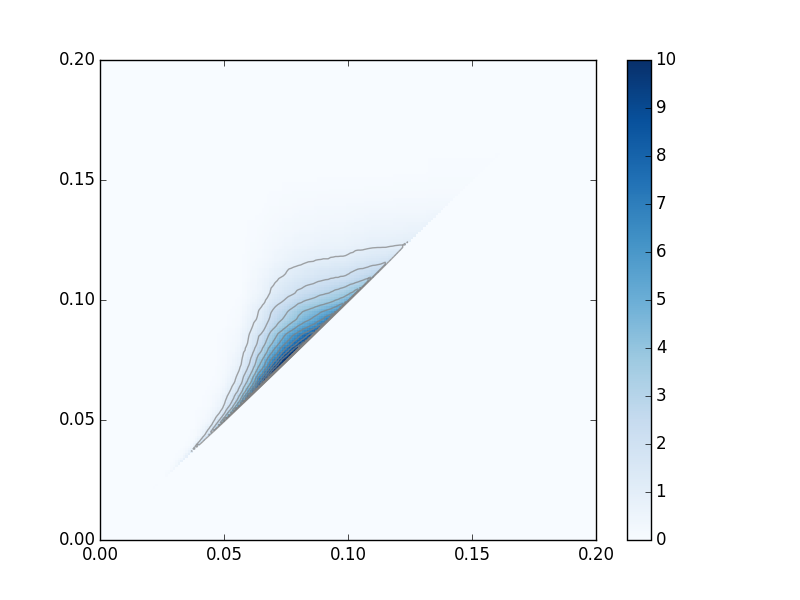}
(b)\includegraphics[height=2in]{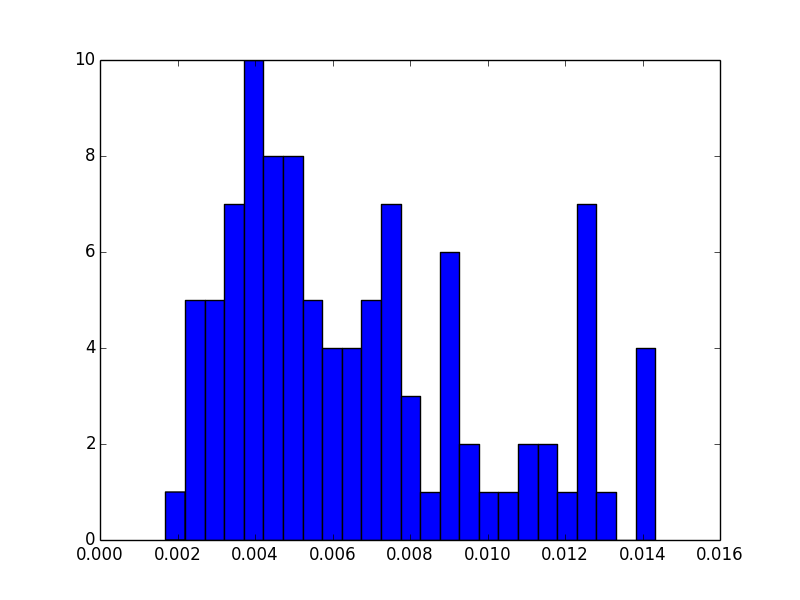}
(c)\includegraphics[height=2in]{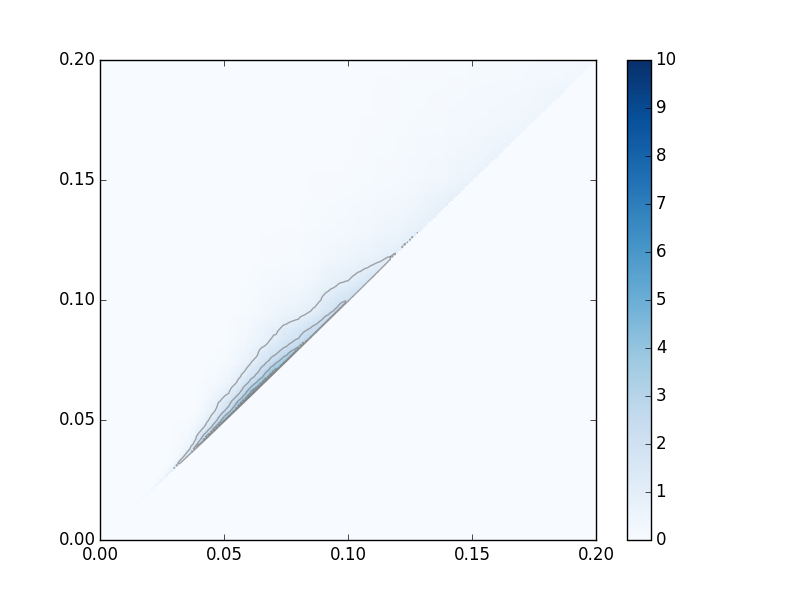}
(d)\includegraphics[height=2in]{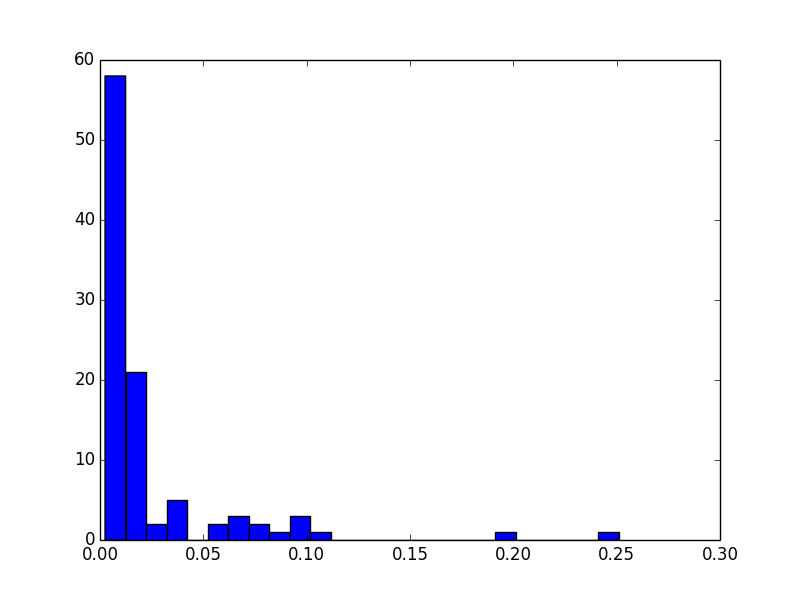}
(e)\includegraphics[height=2in]{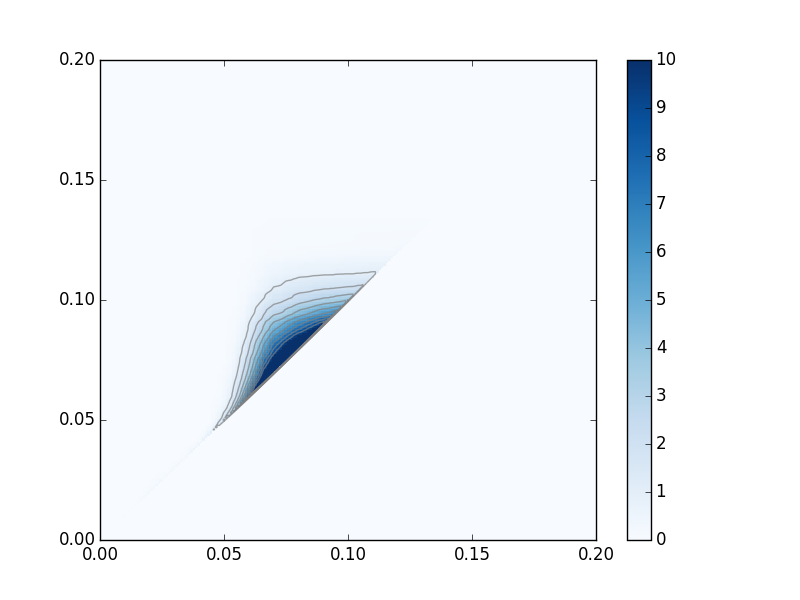}
(f)\includegraphics[height=2in]{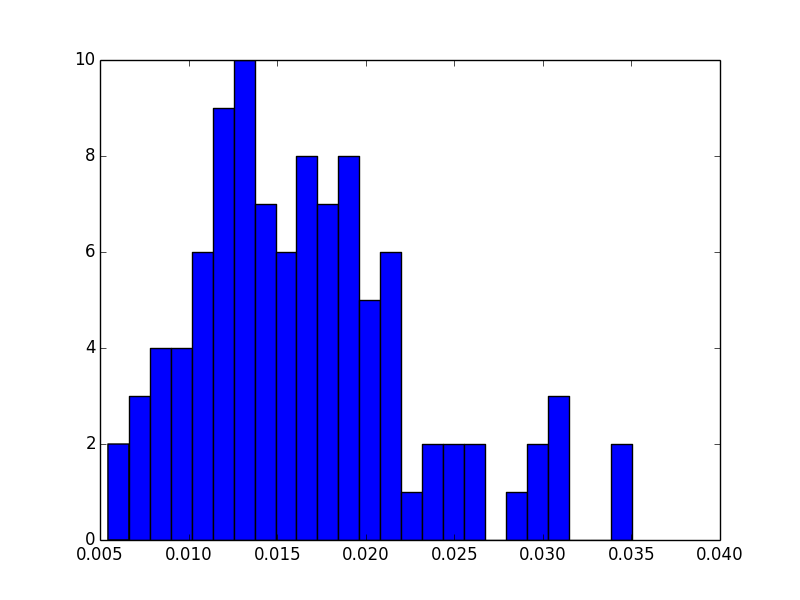}
\caption{(a), (c), (e) are the mean $H_1$ persistent rank function of $100$ i.i.d. spatial point patterns from the models of Strauss, Matern clustering and Baddeley-Silvermann, respectively.
(b), (d) and (f) are the corresponding distributions of distances squared to $\bar{\beta_1}$.}
\end{figure}

The Strauss process with interaction radius $R$ and parameters $\beta$ and $\gamma$ is the pairwise interaction point process with probability density
\begin{displaymath}
f(x_1, . . . , x_n) = \frac{1}{Z}\beta^{n(x)}\gamma^{s(x)}
\end{displaymath}
where  $\{x_1 , . . . , x_n\}$ represent the points of the pattern, $n(x)$ is the number of points in the pattern, 
$s(x)$ is the number of distinct unordered pairs of points that are closer than $R$ units apart, and $\frac{1}{Z}$ is the normalizing constant.  By conditioning on a fixed number of points the parameter $\beta$ is absorbed into the normalizing constant. We used $R=0.05$ and $\gamma=0.5$.

The Matern clustering model we used is as follows. First determine the location of the ``parent'' points over the plane by a Poisson distribution with intensity $10$. Then each parent point is replaced by a random cluster of ``offspring'' points, the number of points per cluster being $Poisson(10)$ distributed, and their positions being placed and uniformly inside a disc of radius $0.02$ on the parent point. We then discard everything outside the unit square. We conditioned on $100$ points in total (by redrawing when not $100$ points).

The Baddley-Silvermann spatial point process model provides a counterexample that CSR can be determined by the distribution of pairwise distances.
Divided the space into $100$ equal rectangular tiles. In each tile, a random number of random points is placed. There are either $0, 1$ or $10$ points, with probabilities
$1/10$, $8/9$ and $1/90$ respectively. The points within a tile are independent and uniformly distributed
in that tile, and the numbers of points in different tiles are independent random integers. Unlike the normal Baddley-Silvermann, we then conditioned on $100$ points in total (by redrawing when not $100$ points).

The Monte-Carlo simulations show that even when the number of points within the point patterns involved is small (here $100$), the rank functions can be used as a test for complete spatial randomness. As could be expected, the power of the test is much higher for repulsive models than clustering models. The power of the tests should improve as the number of points increases.  The intuition behind this claim is twofold. Firstly, a functional version of the central limit theorem would imply that the variation of the corresponding distributions of the persistent rank functions will decrease. In other words the distributions of persistent rank functions of each model will concentrate around the means (both for CSR and the other models) implying both that the cutoff radius will decrease and that the measure of the alternative model in that ball will decrease. Secondly, the effect of the boundary will decrease. This should be particularly true for the Matern Clustering model where inspection shows that there are a significant proportion of larger voids which do not appear as $H_1$ persistent homology classes because they appear near the boundary of the unit square.

\clearpage

\section{Functional Principal Component Analysis}\label{sec:PCA}

In this section we give a general outline of functional principal component analysis and how to compute it using the inner product matrix of a set of centered functions. This is a standard technique but is included here for completeness and for accessibility for those from other disciplines.

Given a set of functions $\{f_1, f_2, \ldots,  f_n\}$, whose mean is zero (constructed by just subtracting the mean from each of the functions in the initial set) we can perform functional principal component analysis. Functional PCA gives a sequence of principal component weight functions $\zeta_1, \zeta_2, \ldots, \zeta_{n-1}$ and principal component weight scores $s_{ij}$. These can be defined inductively. From now on let $\langle , \rangle$ denote an inner product on the space of functions $V$.

The first principal component weight function, $\zeta_1$, is defined to be the norm one function that maximizes $\sum_{i=1}^n \langle \zeta_1, f_i \rangle$. We define $\zeta_j$ to be the function that maximizes $\sum_{i=1}^n  \langle \zeta_j, f_i \rangle$ subject to the constraint that $\zeta_j$ is a norm one function which is orthogonal to $\zeta_1, \zeta_2, \ldots, \zeta_{j-1}$. We define $s_{ij}:=\langle \zeta_j, f_i\rangle$. 

Since the set of functions $f_1, f_2, \ldots f_n$ lies in an $n-1$ dimensional subspace of the space of functions we know $f_i=\sum_{j=1}^{n-1} s_{ij}\zeta_j$. 
However, since the principal component weight functions are ordered to be in the directions of greatest variance, approximating  each function $f_i$ by a linear combination of the first $k$ principal component weight functions provides the best $k$-dimensional representation of the sample functions.  
In other words, we can approximate $f_i$ by the partial sums $\sum_{j=1}^k s_{ij}\zeta_j$. How well these partial sums approximate the original data is often expressed in the proportion of explained variance. If $\lambda_j$ is the variance of the scores $\{s_{ij}\}$ (i.e. $\lambda_j = \sum_{i=1}^n s_{ij}^2$) then total variance is $\sum_{j=1}^{n-1} \lambda_j$ and hence the proportion of variance explained by the projection into the first $k$ coordinates is $$\dfrac{\sum_{j=1}^k \lambda_j}{\sum_{j=1}^{n-1} \lambda_j}.$$


When working with finite dimensional data, PCA proceeds by first creating a matrix $X$ where each row is a data point and each column is a coordinate axis (i.e., a variable).  
The principal components are then the eigenvectors of $X^TX$ ordered by the largest corresponding eigenvalue. 
These eigenvalues are also the variances of the corresponding sets of scores.  A simple adaption of PCA to functional data  represents each function by its values at a fixed, finite number of evenly spaced  coordinates; then each function becomes a row in the matrix $X$ and each column is a coordinate where the functions are evaluated.   
The matrix product $X^TX = (XX^T)^T$ has entries that are approximations to the Lebesgue inner product between functions.  
In this paper, however, we are working with functions whose inner product is defined with respect to a weighting function~(\ref{metric}) so we must use a slightly different adaption of the standard PCA matrix representation, interpreting the matrices as a linear operators as follows. 

Given a set of $n$ rank functions, $\{f_1, f_2, \ldots f_n\}$, define the linear operators 
\begin{align*}
X:L^2(\R^{2+},g)& \to \R^n\\
g&\mapsto (\langle g, f_1\rangle, \langle g, f_2\rangle, \ldots, \langle g, f_n\rangle)
\end{align*}
and
\begin{align*}
X^T:  \R^n&\to L^2(\R^2+,g)&\\
(a_1, a_2, \ldots, a_n)&\to \sum_{i=1}^n a_i f_i.
\end{align*}
A matrix representation of $X X^T$ acts on elements of $\R^n$ with the $(i,j)$ entry being $\langle f_i, f_j\rangle$. 
The non-zero eigenvalues and eigenvectors of $X^TX$ can then be found from the eigenvalues and eigenvectors of $XX^T$ because for $w \in \R^n$,  
\begin{displaymath}
XX^T w = \lambda w \implies (X^T X)X^Tw = \lambda X^T w. 
\end{displaymath} 
Thus, to find the principal component weight functions $\zeta_1, \zeta_2, \ldots, \zeta_n$ we first compute the eigenvalues and eigenvectors $\lambda_1, \ldots \lambda_n$ and $w_1, \ldots w_n$ of the matrix of inner products $(\langle f_i, f_j\rangle)_{i,j =1}^n$ and then set $\zeta_k$ to be the unit-norm function that is a scalar multiple of $X^T w_k$.

\subsection{Outline of algorithm}

We will now outline the process of computing PCA of discretized rank functions from a set of persistence diagrams. We start with persistence diagrams as there are many standard algorithms to compute persistence diagrams. The first step is to transform the persistence diagrams to rank functions. Discretize the plane by choosing a lattice over a relevant region over the plane above the diagonal. This provides a list of coordinates $\{(x_1,y_1), (x_2, y_2),  \ldots (x_M, y_M)\}$. Each rank function will be written as a length $M$ vector.

Given a persistence diagram we want to compute the corresponding rank function. Start with the zero vector of length $M$.  Then iterate through the points in a persistence diagram:  for each point $(a,b)$ in the diagram add $1$ to coordinate $l$ when with $a<x_l\leq y_l<b$. These are all the coordinates is the lower right triangle between $(a,b)$ and the diagonal.

Given the set of discretized persistence rank functions, written as vectors $\{v_1, v_2, \ldots v_N\}$), the algorithm for computing the mean and principal component weight functions is the following recipe; 
\begin{enumerate}
\item Calculate the mean $\bar{v}$ of $\{v_1, v_2, \ldots v_N\}$ coordinate wise (at each sample point $(x,y)$).
\item Construct new centered arrays $\{\tilde{v}_1, \tilde{v}_2, \ldots \tilde{v}_N\}$ by $\tilde{v}_i = v_i - \bar{v}$
\item Calculate the dot product matrix $D = (\langle \tilde{v}_i, \tilde{v}_j \rangle)_{i,j}$ using the weight function as specified in (\ref{metric}).  
\item Perform an eigenvalue/eigenvector decomposition of $D$, ordered by the size of the eigenvalues (largest first). 
\item Let $\lambda_j$ and $w_j=(a_1, a_2, \ldots a_N)$ be an eigenvalue/eigenvector pair of $D$. $\lambda_j$ is the variance of the scores for the $j$th principal component. The $j$th principal component is the unit-normed function in the direction $\sum_{i=1}^N a_i \tilde{v}_i$.  Thus the $j$th principal component is $\zeta_j=\frac{1}{\|\sum_{i=1}^N a_i \tilde{v}_i\|}\sum_{i=1}^N a_i \tilde{v}_i$ where 
$$\|\sum_{i=1}^N a_i \tilde{v}_i\|^2=\sum_{i=1}^N \sum_{l=1}^N a_i a_l D_{i,l}$$
\item We then can then reinterpret $\zeta_j$ as the discretized function values evaluated at the sample points $\{(x_l,y_l)\}_{l=1, \ldots M}$.  
\end{enumerate}

Once we have computed as many principal component functions as desired we then compute the corresponding scores, i.e. projections of the (centered) rank functions onto the $k$ largest principal components. 
These scores are $s_{ij} = \langle \zeta_j, \tilde{v}_i\rangle$, i.e., dot products of the rank functions with the principal component  functions. Alternatively they can be computed directly from the dot product matrix;
\begin{displaymath}
s_{i,j}= \frac{\sum_{m=1}^N a_m D_{m,i}}{\sqrt{\sum_{k=1}^N \sum_{l=1}^N a_k a_l D_{k,l}}}.
\end{displaymath}


\section{Analysis of Colloids}\label{sec:colloids}


We now apply the PCA techniques to two-dimensional point patterns obtained from images of colloids in an experimental model  designed to test theories of melting in 2D crystals~\cite{gasser_melting_2010}.\footnote{
We thank Peter Keim for giving us permission to use this data. }  
The experimental system consists of paramagnetic colloidal particles of diameter a few microns, suspended in water and fixed by gravity to sit at the air-water interface of a hanging droplet.  An applied magnetic field induces a repulsive dipole-dipole interaction between the particles, and the (inverse of) applied field strength plays the role of temperature.  
At strong field strengths the colloids arrange themselves in a hexagonally close-packed crystal with no defects, and as the field strength decreases the crystal melts, passing through two continuous phase transitions to an isotropic liquid phase.  
The intermediate \emph{hexatic} phase retains the long-range 6-fold orientational order of the crystal but not its quasi-long-range  translational order.  
This two stage melting of 2D crystals is described by a theory of Kosterlitz, Thouless, Halperin, Nelson, and Young (KTHNY) that explains it through the dissociation of pairs of defects that form close to the melting point. 

Extensive studies of the point patterns formed by the colloid centres~\cite{dillmann_comparison_2012} have demonstrated that the bond orientational correlation function, and bond orientational susceptibility parameter are sensitive to the transition between hexatic and isotropic liquid phases, as these quantities explicitly test for long-range 6-fold orientational order.  
The transition between hexatic and crystalline phases should in principal be detected by periodicity in the density (or two-point) correlation function, but in practice this is found to be insufficient as 2D crystals have a logarithmic divergence in their displacement autocorrelation function.  
Instead, the melting transition is better detected by the ``dynamic 2D Lindeman parameter'' which is a measure of a particle's displacement with respect to its nearest neighbours as a function of time.  This quantity stays bounded in a crystal (if no grain boundaries are present) but diverges in the hexatic and isotropic liquid phases.  
The work in~\cite{dillmann_comparison_2012} also shows that other quantities that measure purely local structural properties, such as the Voronoi cell shape factor or Minkowski functionals, do not provide a sensitive test of the phase transitions.  

Ideally, it would be useful to find a single order parameter that changes at both phase transitions.  The persistent rank functions of alpha shapes are a possible candidate as they are sensitive both to the local arrangements of particles, but also naturally incorporate higher-order point correlations and non-local topological information.  The data presented here is an initial investigation that illustrates the basic properties of persistent homology signatures of colloidal point patterns.  

The data used are four groups of seven snapshots of around 800 colloidal particles;
 each group is taken from a different effective temperature, $T = 0.013, 0.014. 0.015. 0.016$.  These images were taken from the centre of a droplet that contained around $10^5$ particles in an equilibrium state.  Example point patterns and their 1-dimensional persistence diagrams are shown in Fig.~\ref{fig:colloid_pts_PDs}.  The filtrations are derived from the 2D alpha-shape of the colloid centres.  Samples at $T = 0.013$ and $0.014$ are crystalline but close to the melting transition, while those with $T=0.015$ are in the hexatic phase and $T=0.016$ are isotropic liquid.   

The principal component analysis of the persistent rank functions for these 28 point patterns shows that $96.7\%$ of variation in 0-dimensional homology is explained in the first component, $97.9\%$ in the first two components, and $99\%$ in the first five.  For 1-dimensional homology the variation explained is lower, $86.4\%$, and $92.0\%$ reaching $99\%$ with 15 components.  
The difference in variation explained here possibly reflects the fact that the overall density of colloids does not change, meaning the 0-dimensional homology information is more consistent across the different effective temperatures, while the appearance of dislocation pairs (from 6,6 to 5,7 nearest neighbours) as T increases, means 1-dimensional homology changes more significantly.  
From the plots of principal component scores in Fig.~\ref{fig:colloid_pca_scores}, we see that in both dimensions, the first p.c.~scores cluster the point patterns into their different phases quite effectively.  The point patterns for $T = 0.013, 0.014$ are noticeably closer together than to the other two temperature groups, consistent with them both being in the crystalline phase.  
The second principal component scores in dimension 0 do nothing to help with the clustering, suggesting that the extra variation explained is physically irrelevant.  
The second p.c.~scores in dimension 1 do seem to provide a little more information, and suggest that one each of the point patterns from $T=0.014, 0.015$ might be anomalous in some way.   
A plot of first p.c.~scores from dimension 0 versus dimension 1 (Fig.~\ref{fig:colloid_dim0_dim1_fpc}) shows a roughly linear relationship, and nice clustering again of the different temperature groups.  

In conclusion, the variation in persistence diagrams for these colloid data is difficult to detect by eye, but the PCA analysis has successfully clustered them into their temperature groups.  The basic persistent homology data as presented here appears to change continuously with the effective temperature, and is unlikely to be a sensitive test of the critical phase transition point.  However, it may be possible to derive a quantity from the persistence analysis that will be sensitive, or to make a persistence analysis of a different function derived from the point patterns.  

\begin{figure}[h]
\begin{center}
\includegraphics[width=0.4\textwidth]{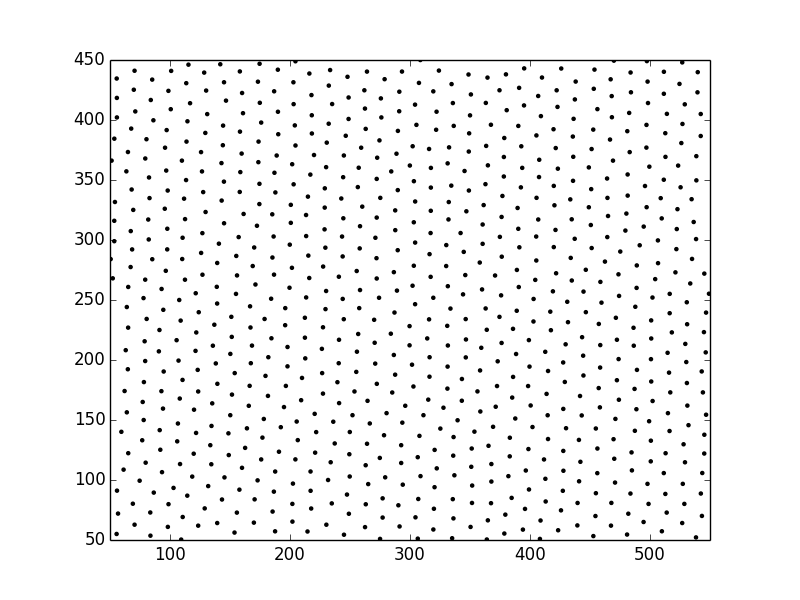}
\includegraphics[width=0.4\textwidth]{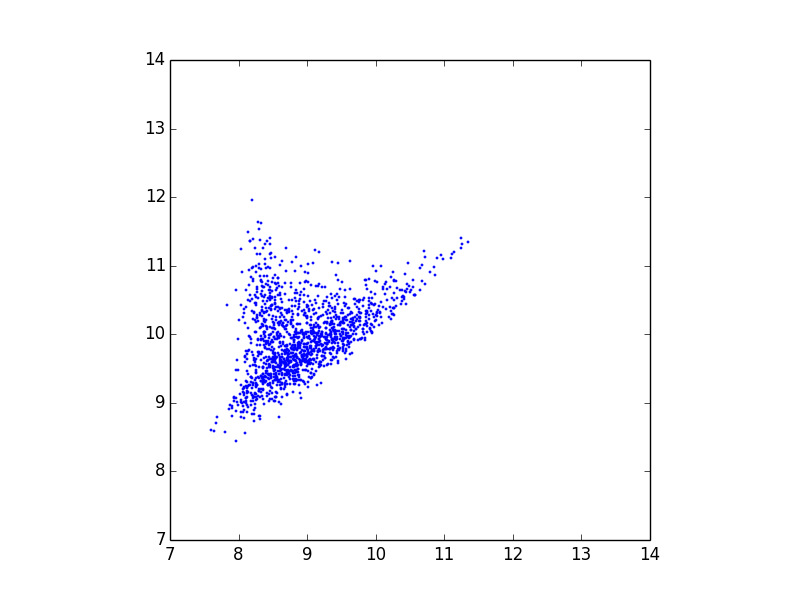} \\
\includegraphics[width=0.4\textwidth]{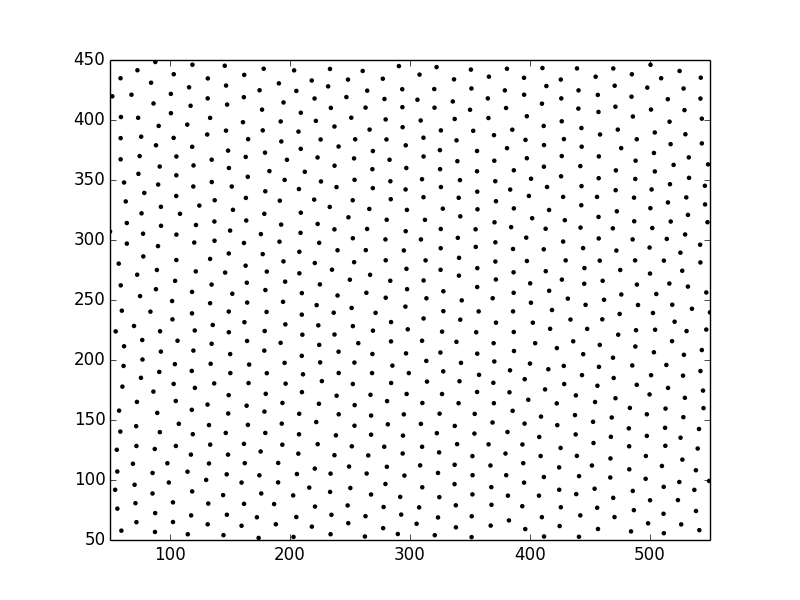}
\includegraphics[width=0.4\textwidth]{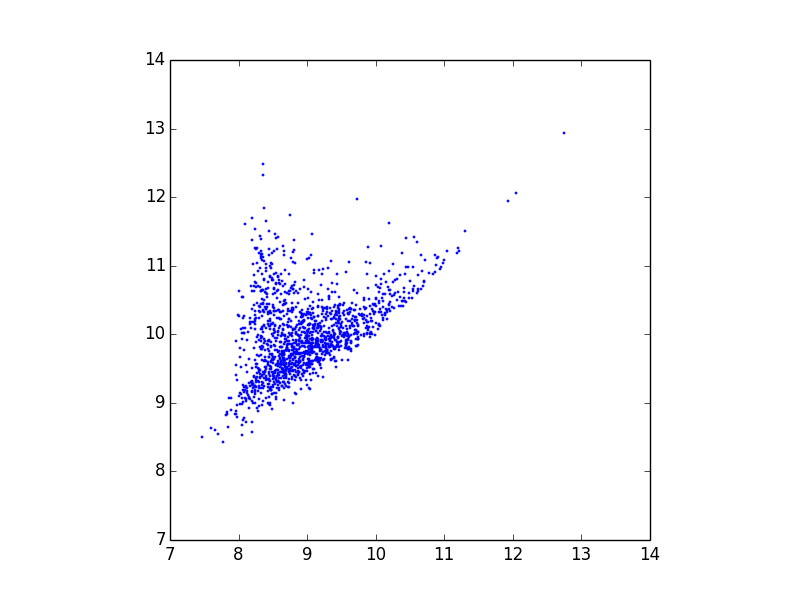} \\
\includegraphics[width=0.4\textwidth]{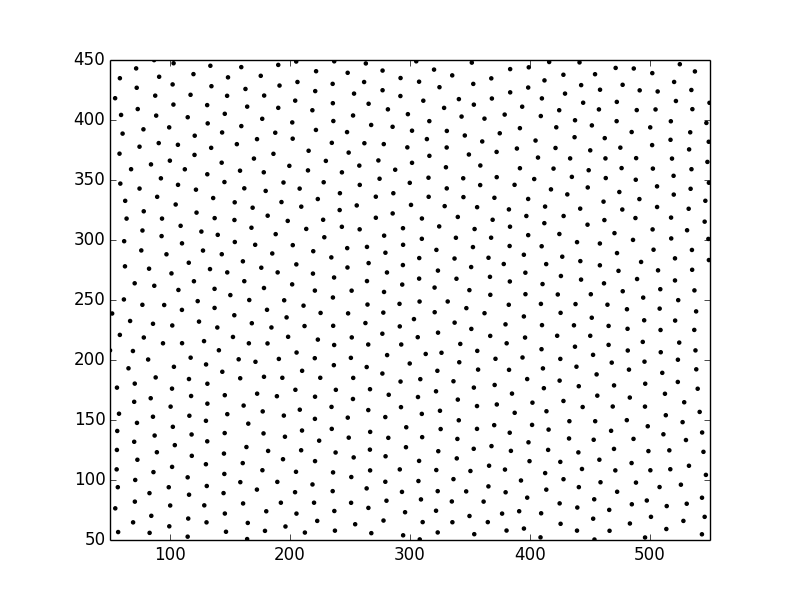}
\includegraphics[width=0.4\textwidth]{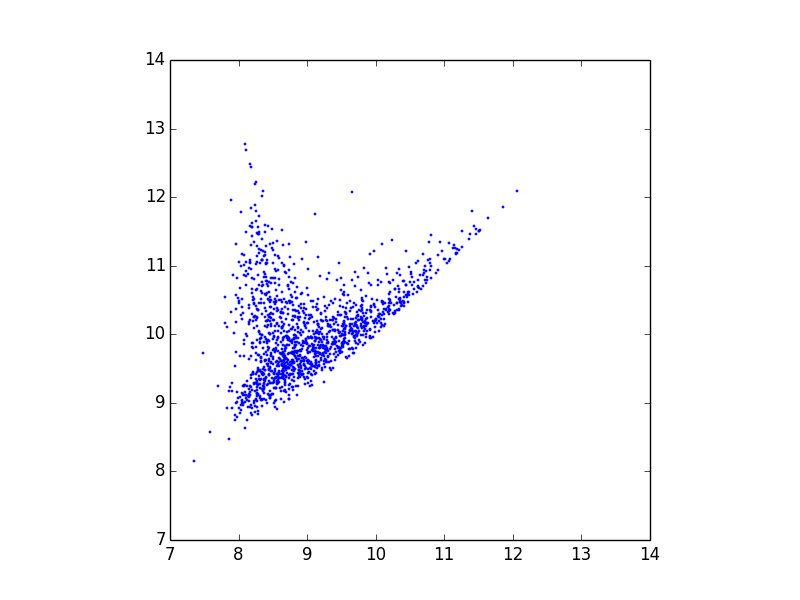} \\
\includegraphics[width=0.4\textwidth]{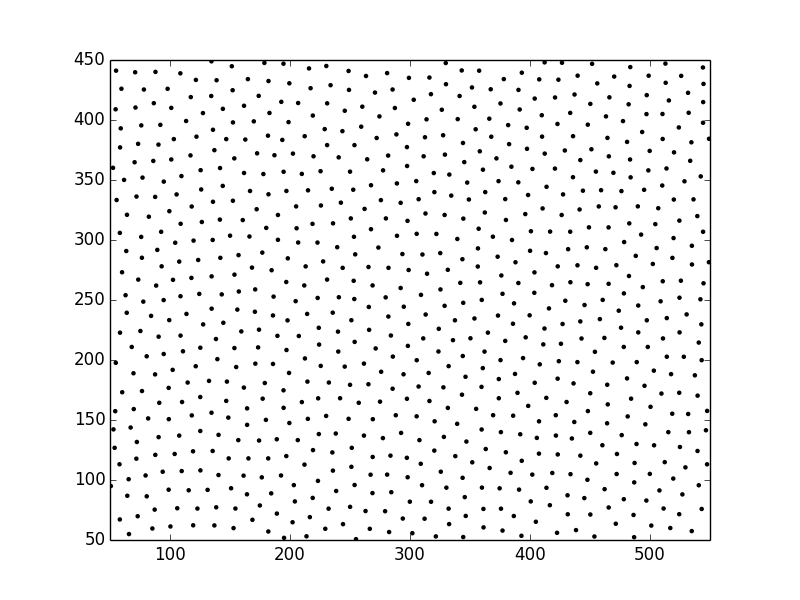}
\includegraphics[width=0.4\textwidth]{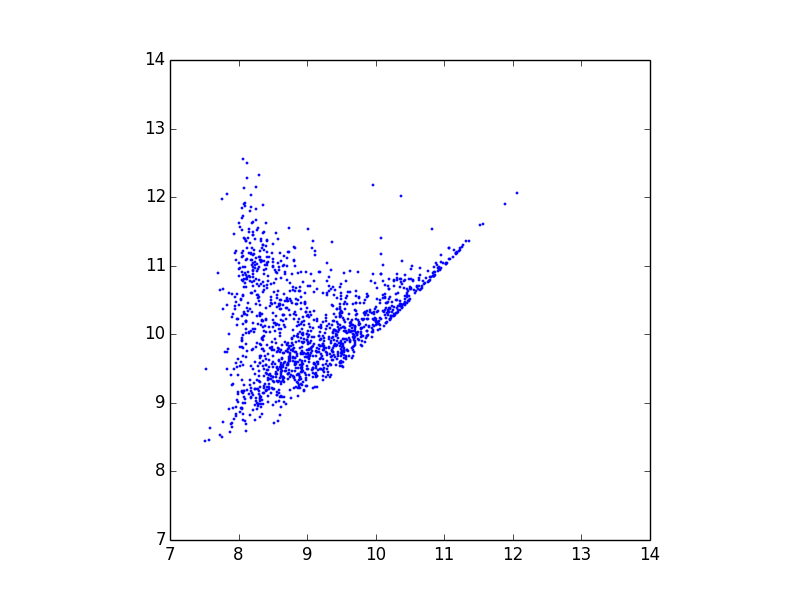} \\
\caption{Point patterns formed by colloidal particles (left column) and their corresponding 1-dimensional persistence diagrams (right column).  One example from each group of seven snapshots is shown, from top to bottom the effective temperature increases from $0.013, 0.014, 0.015. 0.016$.  Units of length are micrometers. }
\label{fig:colloid_pts_PDs}
\end{center}
\end{figure}

\begin{figure}[h]
\begin{center}
(a)\includegraphics[width=0.3\textwidth]{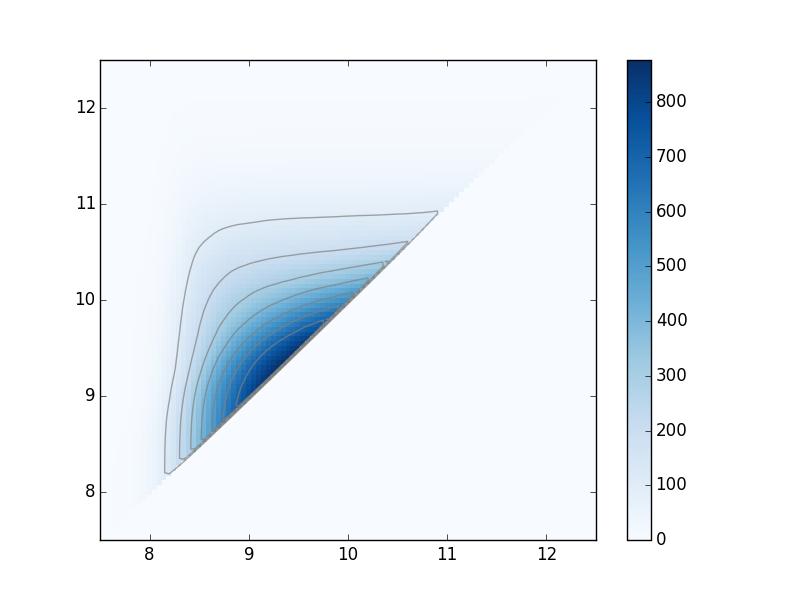}
(b)\includegraphics[width=0.3\textwidth]{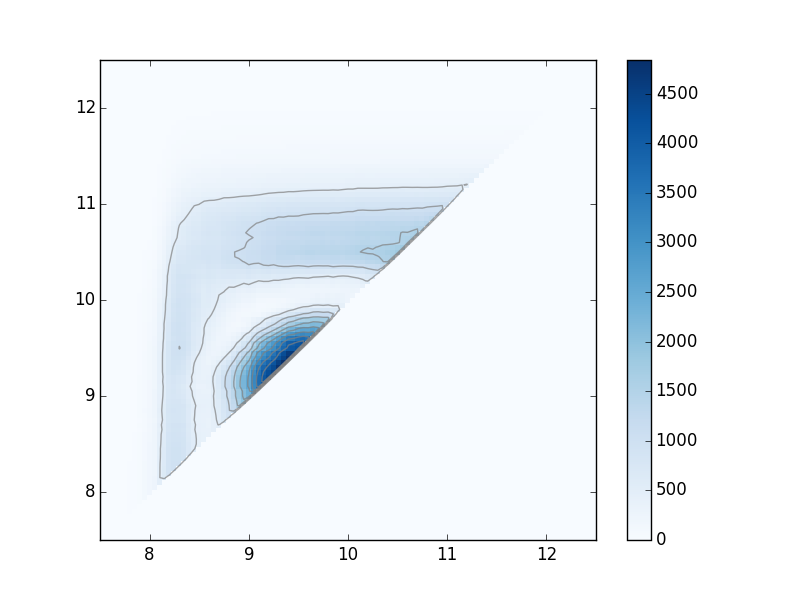} 
(c)\includegraphics[width=0.3\textwidth]{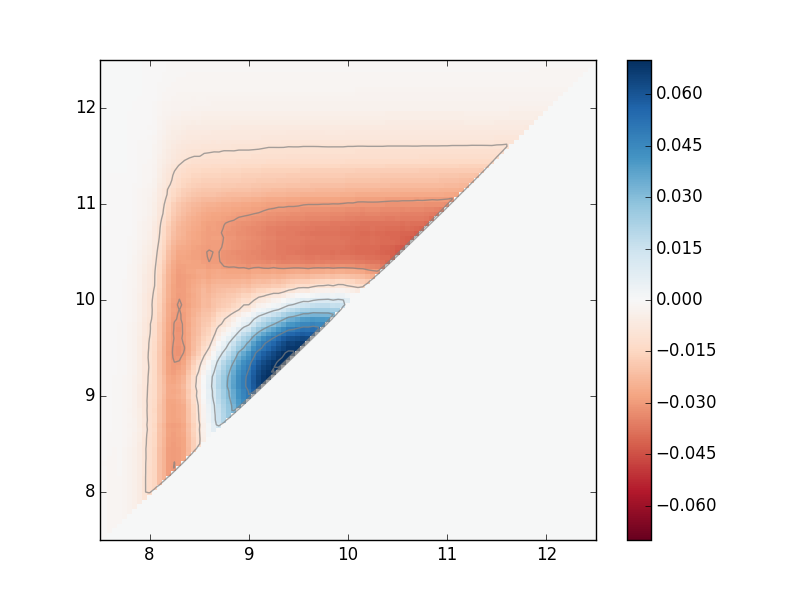}
\caption{(a) Mean rank function for $\beta_1(s,t)$ of 28 colloidal point patterns.  (b) Pointwise variance in the $\beta_1$ rank functions. (c) First principal component function.  Contours in (a) and (b) match the tick marks on the colour bar, contours in (c) are $\pm0.01, \pm 0.03, \ldots$. }
\label{fig:colloid_mean_var_fpc}
\end{center}
\end{figure}

\begin{figure}[h]
\begin{center}
(a)\includegraphics[width=0.45\textwidth]{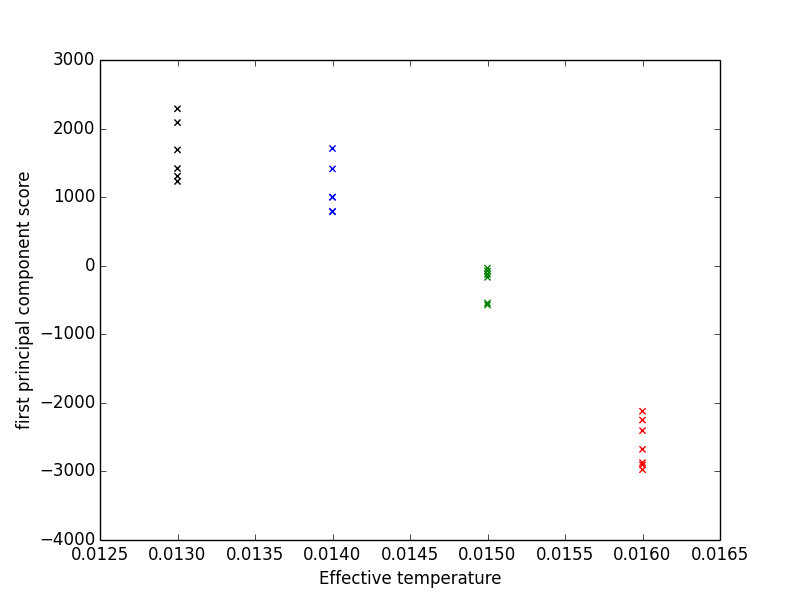}
(b)\includegraphics[width=0.45\textwidth]{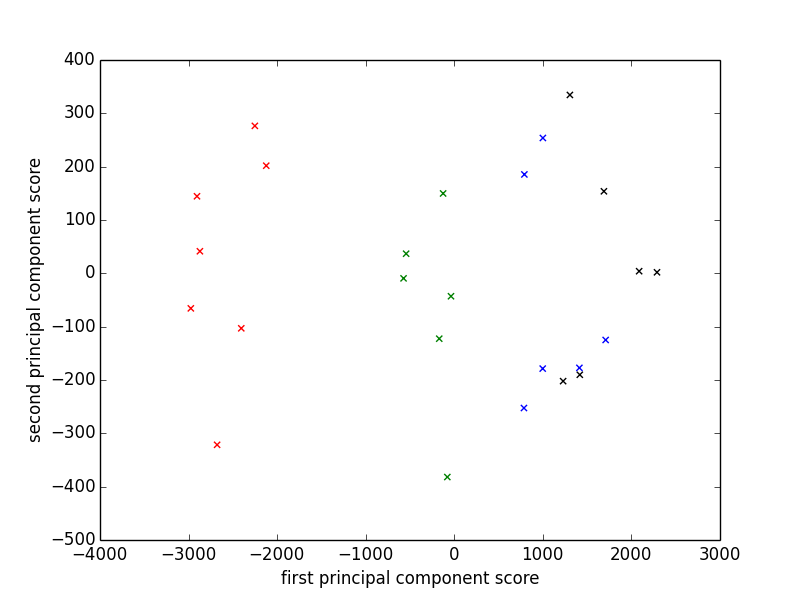} \\
(c)\includegraphics[width=0.45\textwidth]{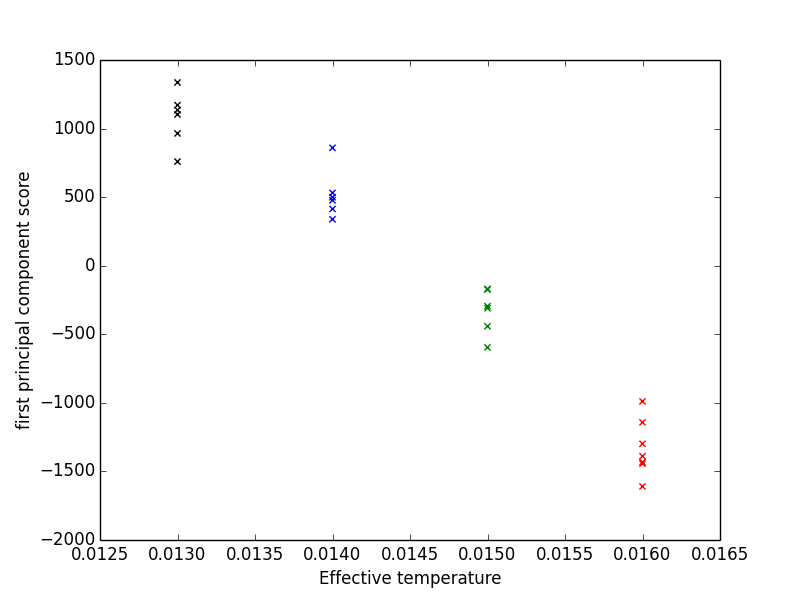}
(d)\includegraphics[width=0.45\textwidth]{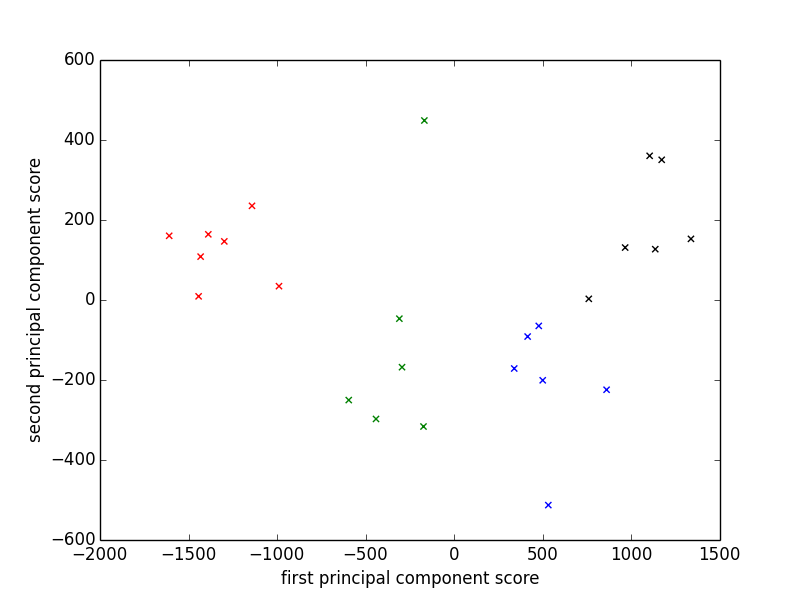}
\caption{Principal component score relationships for the 28 colloidal point patterns.  (a) First principal component scores versus effective temperature for 0-dimensional persistent homology. (b) First versus second p.c. scores for $\beta_0$.  (c) First principal component scores versus effective temperature for 1-dimensional persistent homology. (d)  First versus second p.c. scores for $\beta_1$.  The colours denote the effective temperature groups.}
\label{fig:colloid_pca_scores}
\end{center}
\end{figure}

\begin{figure}[h]
\begin{center}
\includegraphics[width=0.45\textwidth]{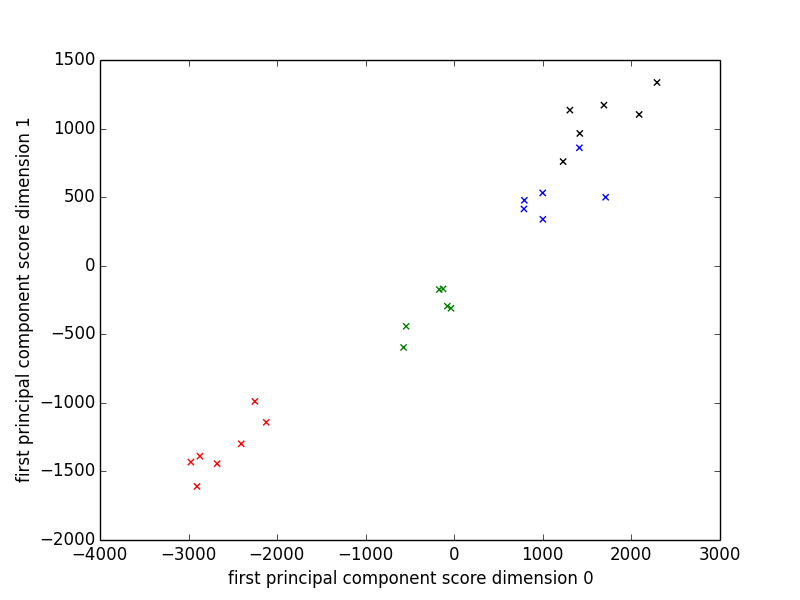}
\caption{First principal component scores in dimension 0 versus dimension 1.  Colour again denotes the effective temperature groups as in Fig.~\protect{\ref{fig:colloid_pca_scores}}. }
\label{fig:colloid_dim0_dim1_fpc}
\end{center}
\end{figure}

\clearpage

\section{Analysis of Sphere Packings}\label{sec:spheres}

In this section we apply the functional PCA methods outlined in Section \ref{sec:PCA} to three-dimensional point patterns obtained from packings of spherical acrylic beads in cylindrical containers.\footnote{
We thank Mohammad Saadatfar for providing us with the data. } 
The points are the centers of the beads and the data sets are non-overlapping cubical subsets interior to the container.  
We find that the rank functions $\beta_1(s,t), \beta_2(s,t)$ computed from alpha shape filtrations each have around $97\%$ of their variation explained by the first principal component and a further $2\%$ when adding the second principal component.  
The scores against the first principal component relate strongly to the packing fraction $\phi$ of the bead-packs.   
This supports the physical assumption that packing fraction is the dominant parameter influencing the configurations of bead packings.  

Monosized spherical bead packings are an idealised model of granular materials.  In the language of statistical physics, they are \emph{athermal}, \emph{dissipative} and \emph{non-equilibrium} systems and so fail to conform to the requirements of standard equilibrium statistical physics models.  Despite this, disordered bead packings do have reproducible distributions of localised quantities that appear to depend on just a few macroscopic parameters, and there have been a number of statistical models since Edwards first proposed that volume could play the role of energy~\cite{edwards_theory_1989}, see the recent review~\cite{bi_statistical_2015}. 
For example, the distribution of localised volumes such as those of Voronoi cells or Delaunay tetrahedra have theoretically-derived distributions depending on the global volume fraction~\cite{aste_invariant_2007}. 

One of the main physical observations about monosized spherical bead packings is that although a maximally dense packing of beads has $\phi = 0.74$~\cite{hales_proof_2005}, there appears to be a random close packing `limit' at $\phi = 0.64$ in the sense that gentle tapping of a container of beads results in this packing fraction, and to obtain a denser packing requires significantly more energy to be injected into the system.  
A maximally dense packing of beads is built from hexagonally close-packed layers that contain local packing configurations of regular tetrahedra and octahedra in a 2:1 ratio, whereas packings with $\phi < 0.64$ are now understood as being built of beads in quasi-regular tetrahedral configurations. 
The geometry of local bead configurations in large packings has been explored extensively using the Voronoi partition and the dual Delaunay complex built from the bead center points.
Results for simulated hard sphere packings in~\cite{anikeenko_shapes_2009,lochmann_statistical_2006} show that packings with $\phi > 0.645$ show a significant increase in the number of Delaunay tetrahedra with quarter-octahedral shape.  
Packings with $\phi \approx 0.64$ have also been shown to have almost all beads belonging to face-sharing quasi-regular tetrahedral configurations~\cite{anikeenko_polytetrahedral_2007,francois_geometrical_2013}. 

In this section, we capture the local configurations of beads via persistent homology diagrams derived from alpha-shape (\v{C}ech)  filtrations of the Delaunay complex.  
This means the analysis is closely related to the above techniques, but the advantage is that persistent homology detects  correlations between adjacent Delaunay tetrahedra in a geometrically and topologically meaningful way.  
For example, rather than testing for the presence of octahedra indirectly by computing a shape parameter to select quarter-octahedra, the persistence information captures an octahedral configuration unambiguously by a point in PD2 with $(b,d)$ = $(1.15, 1.41)$.
The birth value for the PD2 point is the circumradius of a face of the octahedron, i.e., an equilateral triangle, and the death value is the circumradius of the octahedron with vertices at the centers of beads of radius 1.   
This range of radii $(1.15, 1.41)$ is exactly that for which the union of alpha-shape balls enclose a void inside the octahedron. 
For a regular tetrahedron, the corresponding range is $(1.15, 1.22)$.   
Further details about the information contained in persistence diagrams are given below.  
Other recent work applying  persistent homology to study granular packings has considered two-dimensional packings of disks, examining force chains~\cite{kramar_quantifying_2014}, the dependence of Betti numbers on preparation protocol,~\cite{ardanza-trevijano_topological_2014}, and the structure of the full configuration space of a small number of hard disks~\cite{carlsson_computational_2012}.  

The data used in this section were obtained from three experiments imaged with the x-ray micro-CT facility at the Australian National University~\cite{aste_geometrical_2005,francois_geometrical_2013}. 
The experiments each used monosized (within $2.5\%$) acrylic beads with a mean diameter of $1.0$mm that were poured into large cylindrical containers (inner diameter of $66$mm) with over 100,000 beads in each case. 
To create the low-density packing, $A$, with overall packing fraction $\phi = 0.59$, a glass rod was placed in the cylinder as the beads were poured in and then gently removed.  
The second experiment, $B$, imaged a random close packing of beads ($\phi =0.63$) after pouring them into the container and then tapping gently to settle. 
In the third experiment, $C$, a loose lid was placed on top a random close packing of beads, the container was shaken intensely for a few seconds to the point of fluidisation and partial crystallisation was observed with resulting $\phi = 0.70$.  
Each container of beads was imaged using micro-CT and the resulting volume maps of x-ray density were processed to extract the coordinates of each bead centre and its radius to within $1.5 \mu$m. 

Our first step in analysing this data is to scale the units of length in each experiment by the mean bead radius of the sample.
We then take disjoint cubical subvolumes of equal sizes from the three experimental data sets and compute the persistence diagrams of alpha-shape filtrations derived from the 3D point patterns of bead centres using the dionysus package~\cite{morozov_dionysus_2013}.  
Persistence diagrams for subsets of between 3000 to 3800 beads from each experiment are shown in Fig.~\ref{fig:beadpack_PD_histograms}.  
These diagrams illustrate a number of interesting features in the geometry of bead packings. 
We begin by noting that the persistence diagrams for 0-dimensional homology simply record the number of beads in the sample and the fact that they are all in contact at the mean bead radius. 
The mean bead radius does not vary significantly across the different subsets so we do not use the PD0 data further.  
PD1 points tell us about the geometry of shortest cycles: birth values record the longest edge-length in a cycle of bead contacts or near-contacts, the death value of each cycle is the circumradius of the largest empty triangle within the span of that cycle.  
The PD1 diagrams in Fig.~\ref{fig:beadpack_PD_histograms} have three main features.  The first is a ``hot spot'' around $(b,d) = (1.0, 1.15) = (1, \tfrac{2}{\sqrt{3}})$, corresponding to an equilateral triangle configuration of beads.  Then there are two approximately one-dimensional features extending from the hot spot, one with $b \approx 1.0, d < 1.4$ and the other along an arc reaching up to $(b,d) = (1.41, 1.41)$, the signature of a right-angled triangle with hypotenuse equal to $2\sqrt{2}$.  
Cycles with $b \approx 1.0$ are due to beads in contact. Given the $2.5\%$ polydispersity in the bead radii, we conclude that those PD1 points with  $b \approx 1.0$ and $d> 1.18$ must be due to (minimal) cycles with four or more beads in contact.  Such cycles have a significant presence in the random loose and close packings, $A, B$, but are dramatically reduced in the densest packing, $C$.  
The cycles along the arc from $b = 1.0$ to $1.41$ are due to triangular configurations with longest edge $= 2b$.  Again, these have a  strong signature in $A$ and $B$, but not in the densest packing $C$.   
The dense packing however, has a second hot spot at $(b,d) = (1.41, 1.41)$ due to right-angled triangles in octahedral bead configurations.  
The presence of octahedral bead configurations is shown extremely clearly in the PD2 diagram of the dense packing, $C$.  
This diagram has three hot spots: one at the tetrahedral point $(1.15, 1.22) = (\tfrac{2}{\sqrt{3}}, \tfrac{\sqrt{6}}{2})$, one at the regular octahedral point $(1.15, 1.41) = (\tfrac{2}{\sqrt{3}}, 2\sqrt{2})$ and one at the right-angled tetrahedral point $(1.41, 1.41)$. 
PD2 diagrams of the random bead packings are significantly different from that of the dense packing.  
The random packings, $A$, $B$, show an arc of PD2 points extending from the regular tetrahedron up to the right-angled tetrahedron that can be bounded below by tetrahedral configurations with a single long edge, and bounded above by tetrahedra with two long edges (opposite one another).  
Regular octahedral configurations are not present in the random packings, but there is a diffuse cloud of points with $1.2<b<1.4$, $d< 2.0$.  
These persistence points are related to the less dense regions of the packings.  
A comprehensive analysis of the information encoded in the persistence diagrams of real bead packings is given in another paper~\cite{saadatfar_persistent_2015}.  

\begin{figure}[h]
\begin{center}
(a)\includegraphics[width=0.3\textwidth]{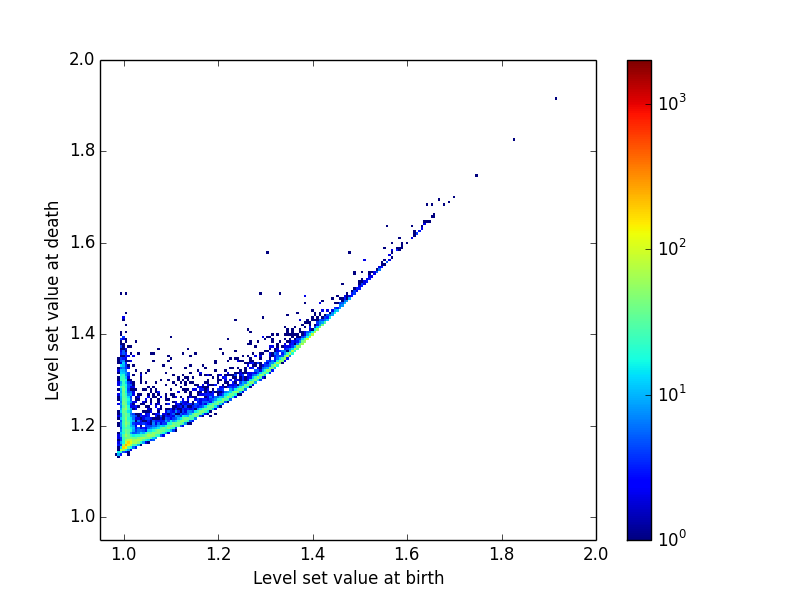}
(b)\includegraphics[width=0.3\textwidth]{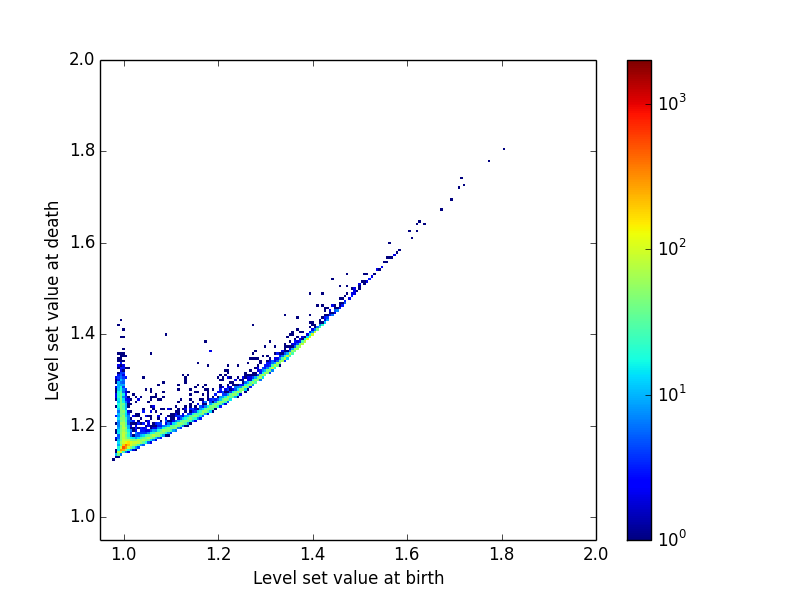}
(c)\includegraphics[width=0.3\textwidth]{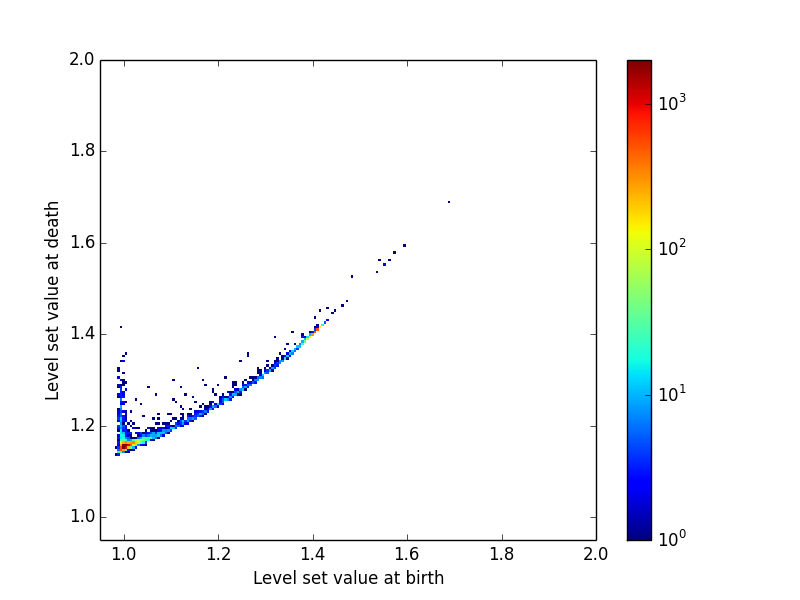}\\
(d)\includegraphics[width=0.3\textwidth]{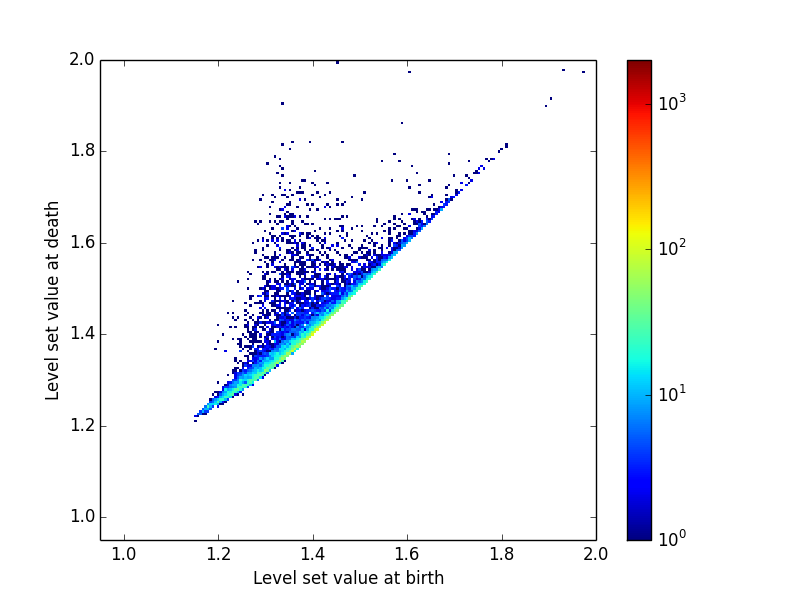}
(e)\includegraphics[width=0.3\textwidth]{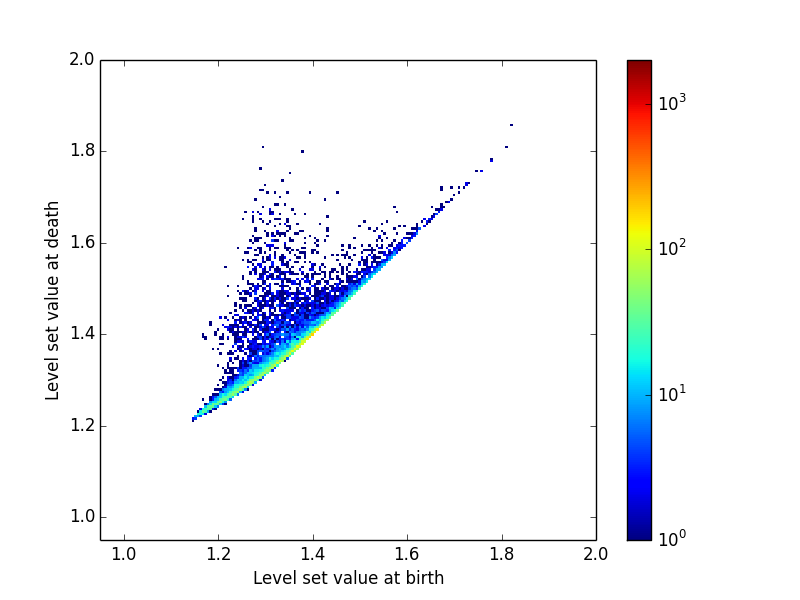}
(f)\includegraphics[width=0.3\textwidth]{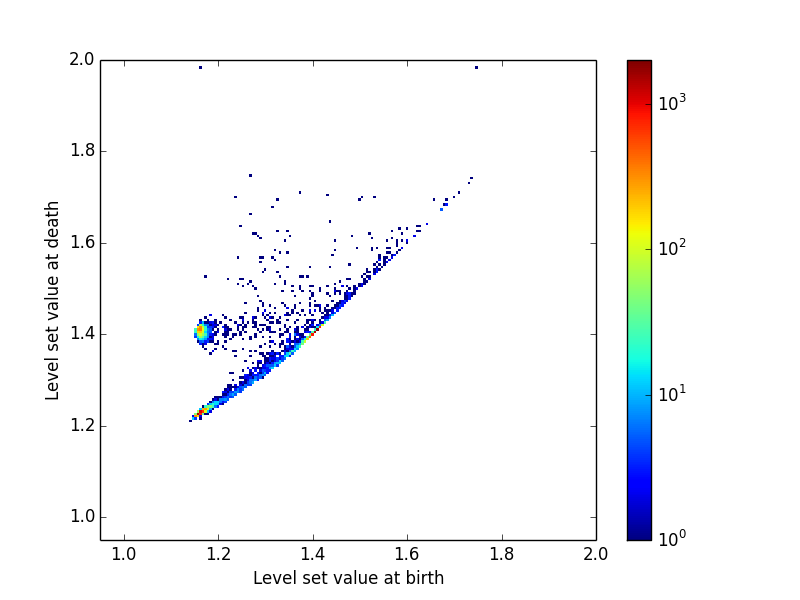}
\end{center}
\caption{Persistence diagram histograms from a $28.0^3$ subset of each experiment.  (a),(b),(c) PD1 for subset volume fractions, $\phi = 0.58, 0.63, 0.72$. (d),(e),(f) PD2 for the same subsets.}
\label{fig:beadpack_PD_histograms}
\end{figure}

To test the apparent dependence of persistence diagrams on volume fraction, we computed 1- and 2-dimensional rank functions for 12 cubical subsets with an edge length of 28.0 units from each of the three experiments, calculated their pointwise mean and variance and performed PCA. 
Mean and variance functions for $\beta_1(s,t)$ and $\beta_2(s,t)$ per unit volume are shown in Fig.~\ref{fig:beadpack_mean_variance_rank_functions} and again highlight the importance of the tetrahedral and octahedral bead configurations. 
The large number of equilateral triangle configurations generates a sharp corner at $(s,t) = (1.0, 1.15)$ in both the mean and variance of $\beta_1(s,t)$, and the regular tetrahedral and octahedral configurations produce significant steps in the mean and variance of $\beta_2(s,t)$ at $(1.15, 1.22)$ and $(1.15, 1.41)$ respectively.

\begin{figure}[h]
\begin{center}
(a)\includegraphics[width=0.4\textwidth]{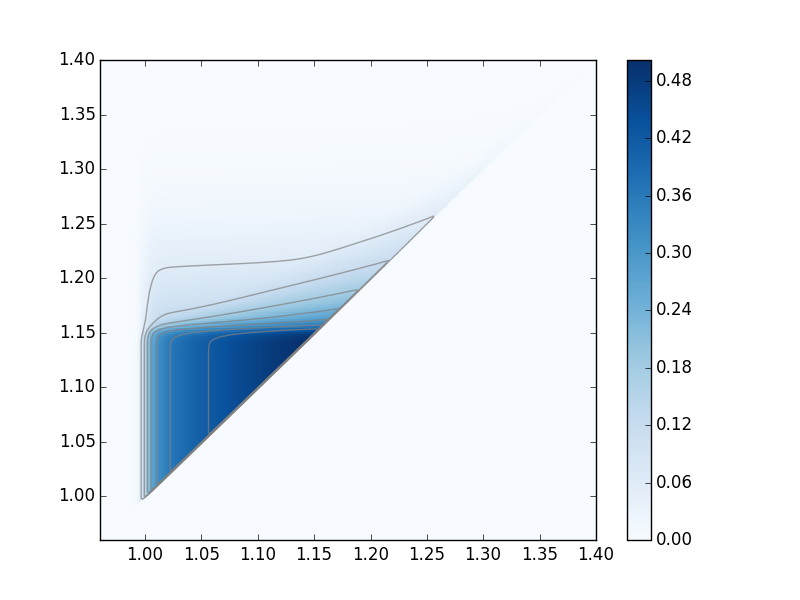}
(b)\includegraphics[width=0.4\textwidth]{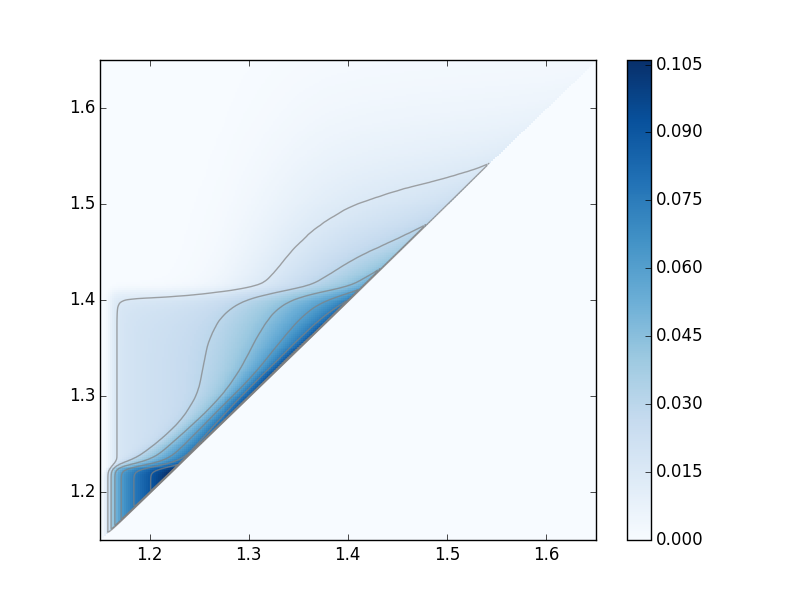} \\
(c)\includegraphics[width=0.4\textwidth]{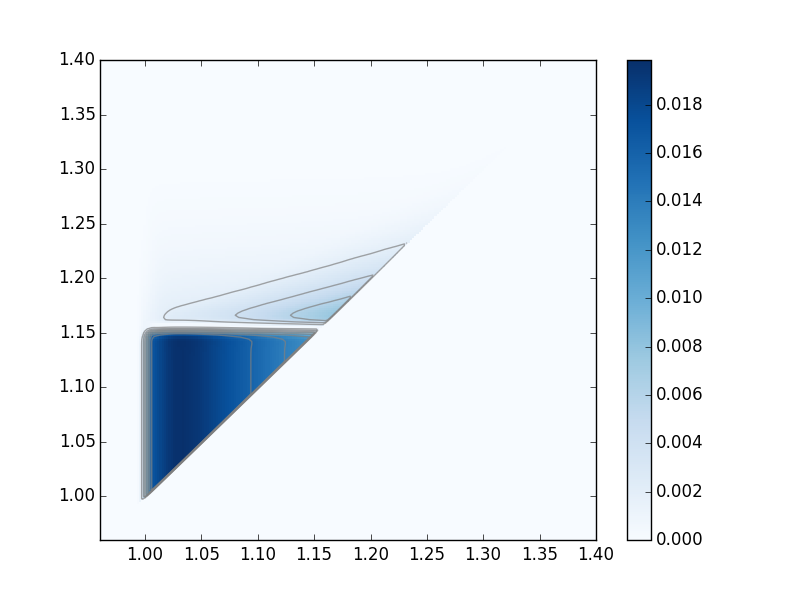}
(d)\includegraphics[width=0.4\textwidth]{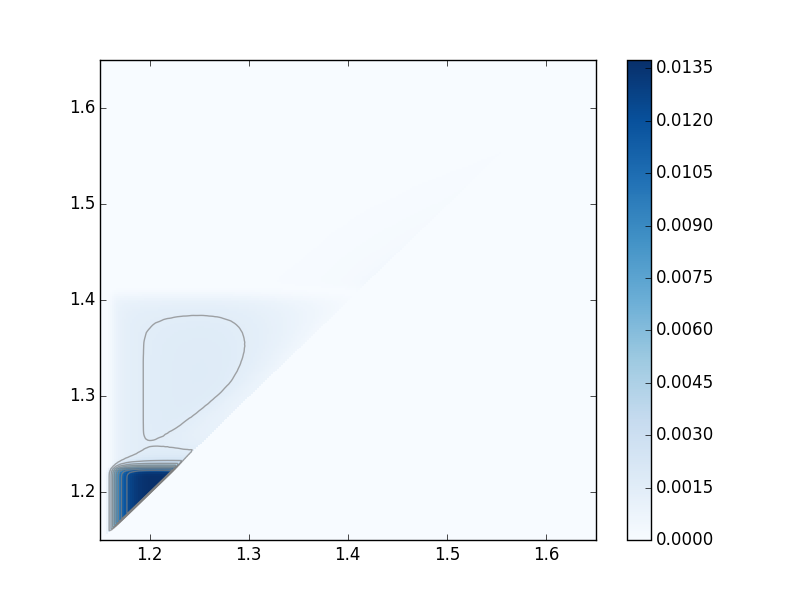} 
\end{center}
\caption{Mean and variance of 1- and 2-dimensional persistent homology rank functions per unit volume from 36 cubical subsets with side $l = 28.0$.   (a) Mean rank for $\beta_1(s,t) / l^3$, (b) variance function for $\beta_1(s,t) / l^3$, (c) mean rank for $\beta_2(s,t) / l^3$, (b) variance function for $\beta_2(s,t) / l^3$.  Contour lines match the tick marks on the colour scale in each plot. }
\label{fig:beadpack_mean_variance_rank_functions}
\end{figure}

The PCA analysis of these 36 subsets shows that $97.7\%$ and $97.1\%$ of the variation in the 1- and 2-dimensional rank functions is explained by a single axis, $99.2\%$ and $99.4\%$ by two-dimensional subspaces respectively.  
The first and second principal component functions are illustrated in Fig.~\ref{fig:beadpack_first_second_princ_comps}.  
Again, the first principal component functions highlight the importance of the regular tetrahedral and octahedral configurations, while the second principal component functions appear to be correcting areas where the first p.c. functions change rapidly.  
In particular the second p.c. functions seem to be related to the variation of regularity amongst the triangles (in dimension 1) and amongst the tetrahedra and octahedra (in dimension 2) which may be partly explained by the polydispersity in bead radii.

We project each subset rank function onto these two principal components and see that the scores cluster according to the subset  volume fractions, see Fig.~\ref{fig:beadpack_PCA_scores}.
The first p.c. scores show a particularly strong relationship with packing fraction.   
The results for $\beta_2(s,t)$ in particular suggest a sharp transition at the Bernal density of $\phi = 0.64$.  
The second p.c. scores are not directly related to the subset volume fractions, but do show some trend within each of the three experiments.

\begin{figure}[h]
\begin{center}
(a)\includegraphics[width=0.4\textwidth]{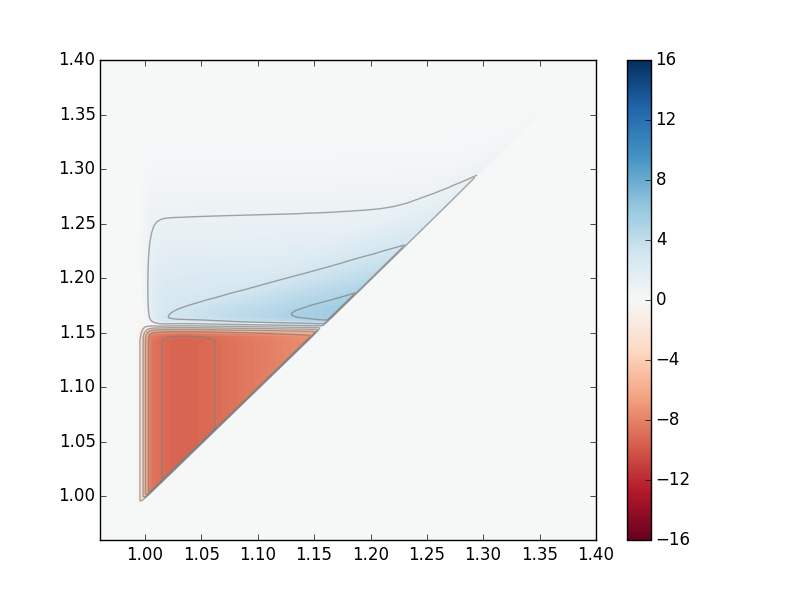}
(b)\includegraphics[width=0.4\textwidth]{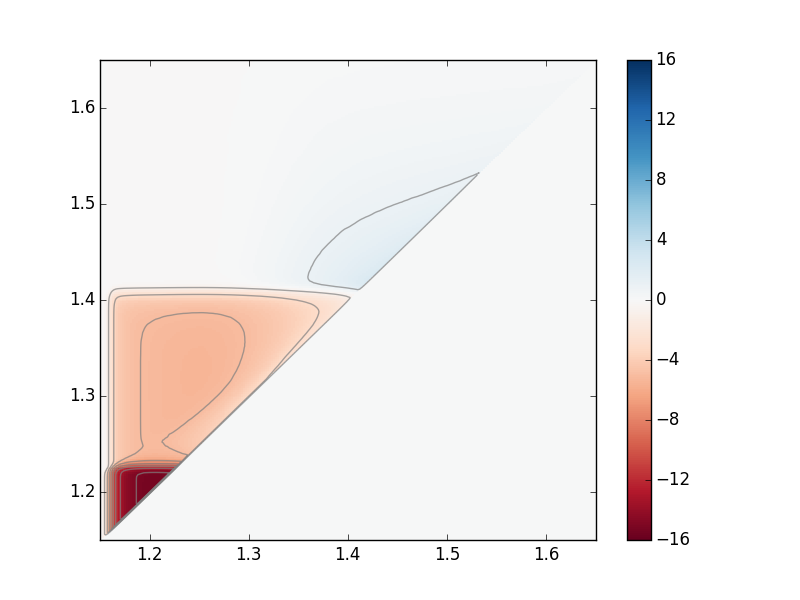} \\
(c)\includegraphics[width=0.4\textwidth]{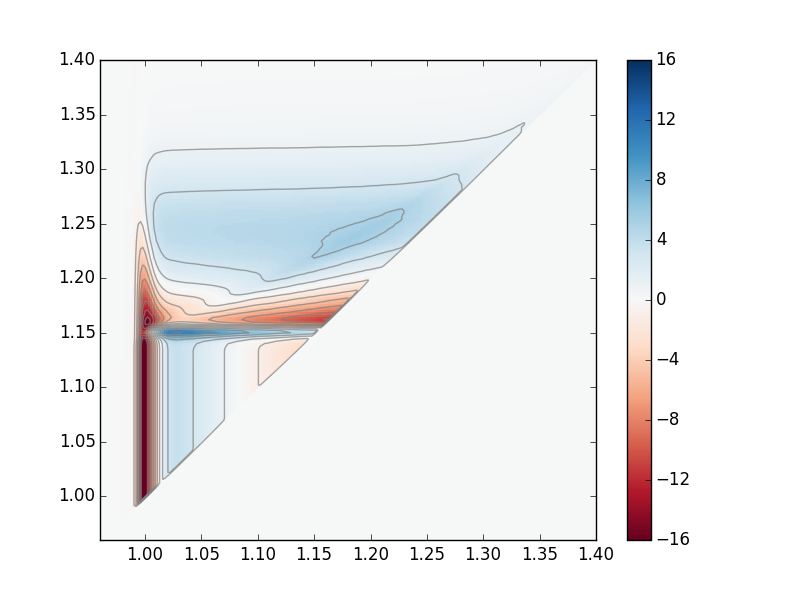}
(d)\includegraphics[width=0.4\textwidth]{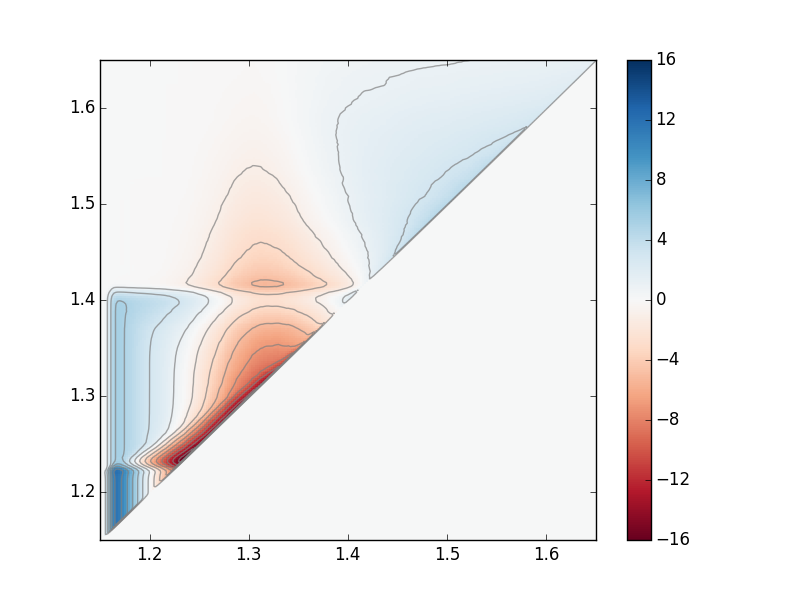} 
\end{center}
\caption{Depiction of the first and second principal component functions for 1- and 2-dimensional persistent homology rank functions from 36 cubical subsets with side $28.0$. (a) First principal component for $\beta_1$, (b) first principal component for $\beta_2$, (c)  
second principal component for $\beta_1$, (d) second principal component for $\beta_2$. The principal component functions are scaled to have their inner product norm = 1.  
Contours are drawn at levels $h = -15, -13, \ldots, -1, 1, \ldots, 13, 15$ in each plot.}
\label{fig:beadpack_first_second_princ_comps}
\end{figure}

\begin{figure}[h]
\begin{center}
(a)\includegraphics[width=0.4\textwidth]{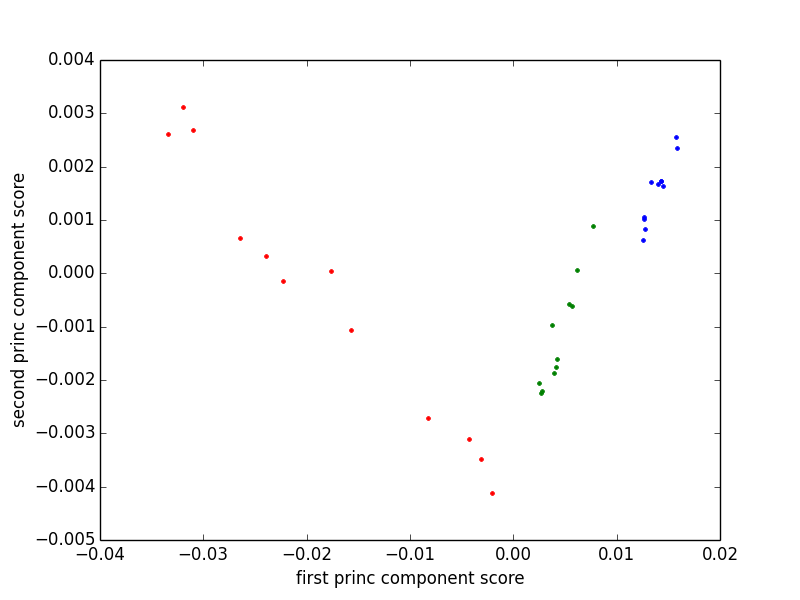}
	(d)\includegraphics[width=0.4\textwidth]{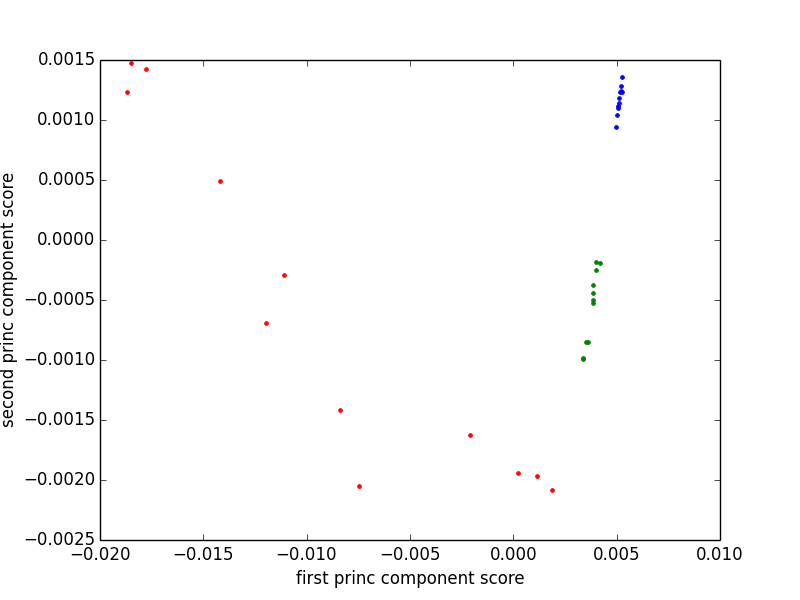} \\
(b)\includegraphics[width=0.4\textwidth]{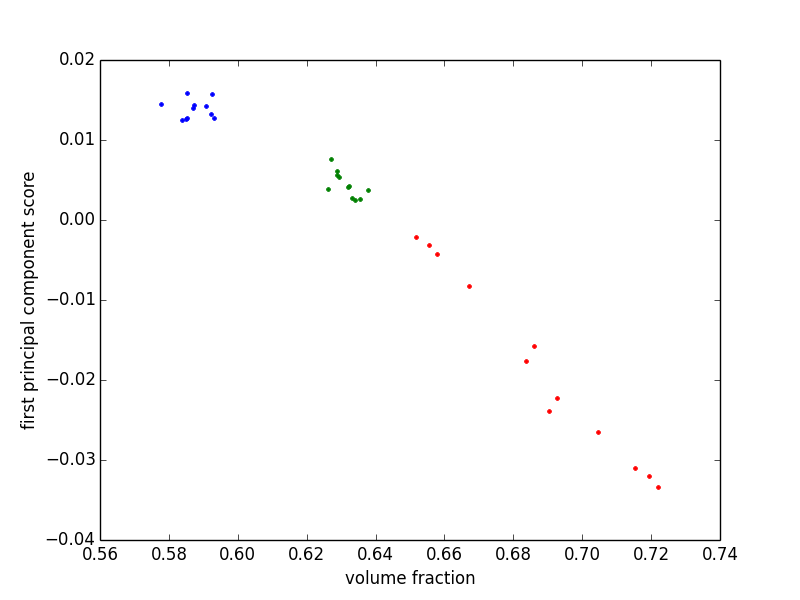}
	(e)\includegraphics[width=0.4\textwidth]{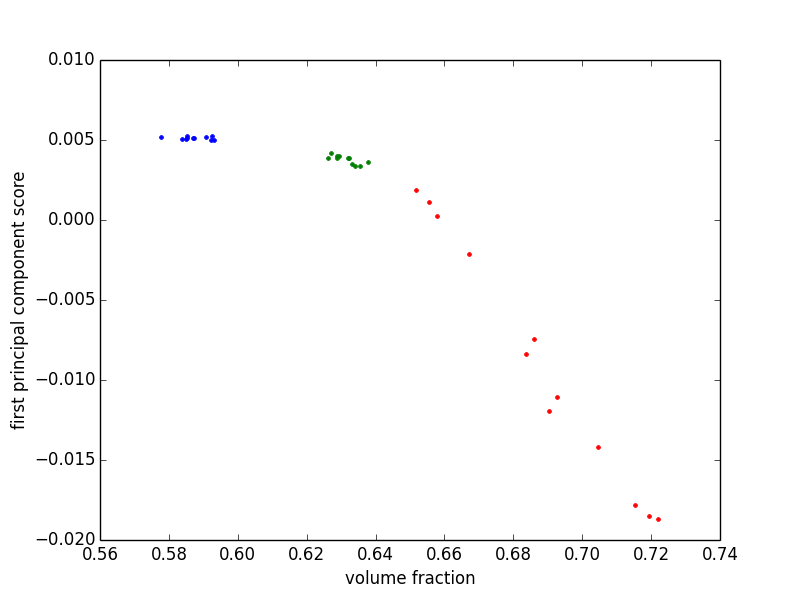} \\
(c)\includegraphics[width=0.4\textwidth]{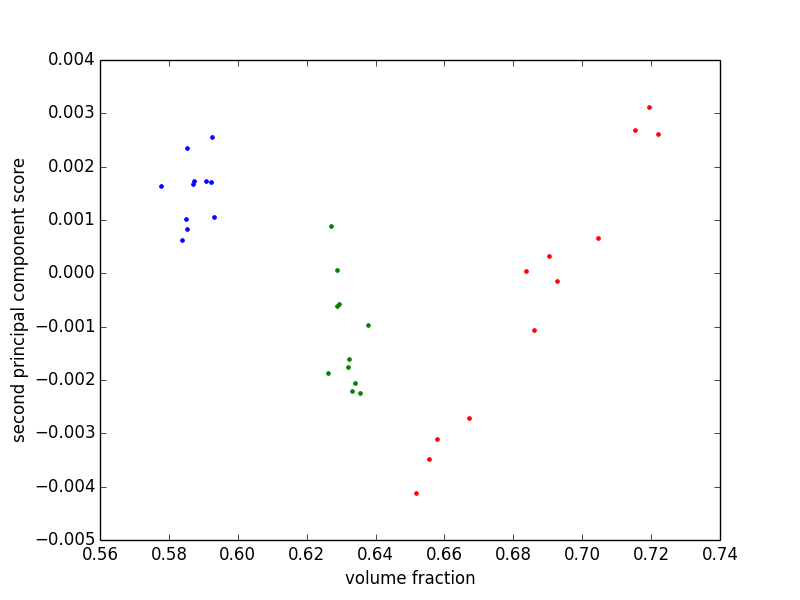}	
	(f)\includegraphics[width=0.4\textwidth]{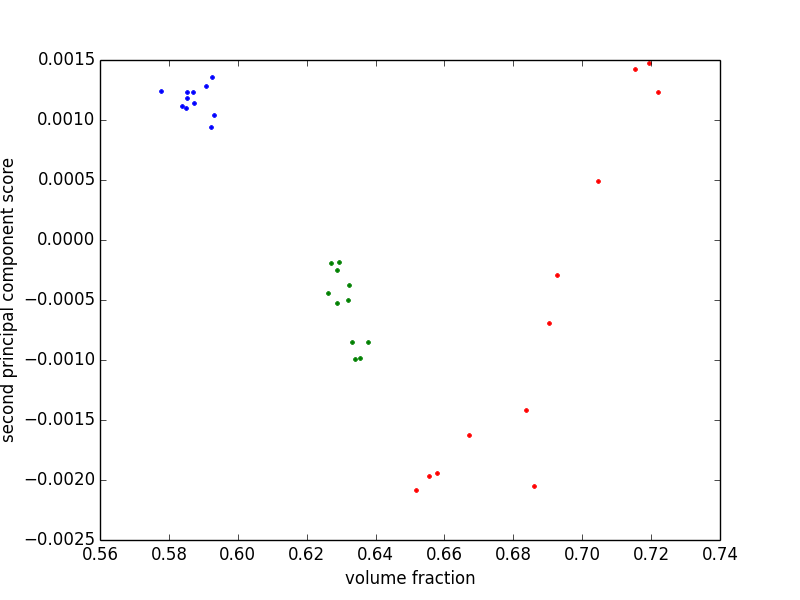}
\end{center}
\caption{ First and second principal component scores for the 36 rank functions. 
(a) First versus second p.c. scores for $\beta_1(s,t)$.  (b) Volume fraction versus first p.c. score for $\beta_1(s,t)$. (c) Volume fraction versus second p.c. score for $\beta_1(s,t)$.  
(d) First versus second p.c. scores for $\beta_2(s,t)$.  (e) Volume fraction versus first p.c. score for $\beta_2(s,t)$. (f) Volume fraction versus second p.c. score for $\beta_2(s,t)$.  
Blue dots are data for subsets from experiment $A$, green dots from $B$ and red dots are from the third, partially crystallised, packing $C$.}
\label{fig:beadpack_PCA_scores}
\end{figure}

The next step in our analysis is to use the structure in the principal component functions to derive simpler order parameters for the bead packings.  
The essential idea is that the locations of the maxima or (negative) minima in the principal component functions indicate the most important geometric parameters for capturing the variation in the rank functions.   
We start with the first principal component for $\beta_2(s,t)$, Fig.~\ref{fig:beadpack_first_second_princ_comps}(b). 
The maximum value occurs as a fairly flat triangular region below the tetrahedral point $(1.15, 1.22)$.
We tested the correlation between $\beta_2(s,t)$ and first p.c. scores for values of $(s,t)$ in this region and found the best linear relationship between them when $(s,t)$ = $(1.20, 1.20)$, see Fig.~\ref{fig:beadpack_analysis}. 
Geometrically, the value of $R_2 = \beta_2(1.20,1.20)$ simply counts the number of enclosed voids in the alpha-shape at radius = 1.20. 
For the bead packings, this means it counts configurations of beads that are close to being regular tetrahedra and regular octahedra.  
This parameter, $R_2$ is therefore similar to the various measures of ``tetrahedrisity'' and ``quartoctahedrisity'' discussed in~\cite{anikeenko_shapes_2009}.    
Our analysis confirms that these quantities are the most important measures of variation in the local configurations of bead packings.  

The second principal component for $\beta_2(s,t)$ has a maximum in a narrow strip around $s = 1.16$, $1.16< t < 1.22$ and a minimum near $s=t=1.235$.  
This suggests an appropriate index is the difference $D_2 = \beta_2(1.167, 1.167) - \beta_2(1.235,1.235)$. 
As shown in Fig.~\ref{fig:beadpack_analysis}, this index has two linear scaling relationships with second p.c. scores, for bead packs above (subsets of packing $C$) and below (subsets of $A$ and $B$) the Bernal limit, $\phi = 0.64$.  
The value of $\beta_2(1.167, 1.167)$ counts the tetrahedra and octahedral configurations that are very close to being regular in the sense that all their faces have circumradii within $1\%$ of that for an equilateral triangle.  
The alpha-shape at radius $1.235$ has filled in all the regular tetrahedra, so $\beta_2(1.235,1.235)$ counts some quasi-regular tetrahedra, the regular octahedra and some quasi-regular octahedral configurations.  
The difference index $D_2$ therefore counts the highly regular tetrahedra, minus some quasi-regular tetrahedral and octahedral configurations.  
The physical significance of this index is unclear, but the two different linear scaling regimes suggests there are different processes relevant to the packings above and below the Bernal limit.  

A similar analysis of the 1-dimensional rank functions supports the findings above.  
The first principal component for $\beta_1(s,t)$ has a fairly flat maximal region below the equilateral triangle point at $(1.0, 1.15)$. 
The values of $R_1 = \beta_1(1.05,1.05)$ have a clear linear relationship with the first p.c. scores, see Fig.~\ref{fig:beadpack_analysis}. 
The quantity $R_1$ is analogous to the 3-point correlation function analysed in~\cite{lochmann_statistical_2006}. 
The structure in the second p.c. for $\beta_1(s,t)$ again suggests an appropriate index could be the difference 
$D_1 = \beta_1(1.028,1.150) - \beta_1(1.00, 1.13) $.  
 As for $D_2$, we see multiple roughly linear scaling relations for subsets from different experiments.  

\begin{figure}[h]
\begin{center}
%
(a)\includegraphics[width=0.4\textwidth]{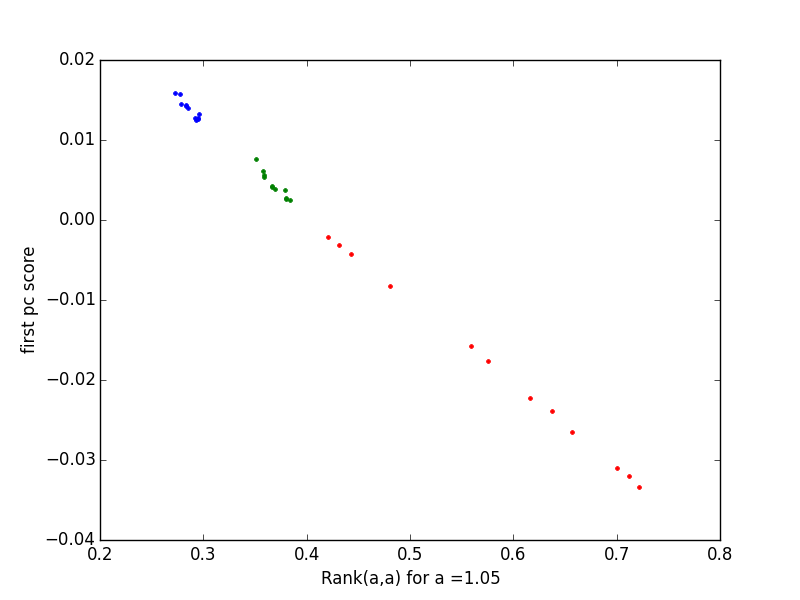}
	(d)\includegraphics[width=0.4\textwidth]{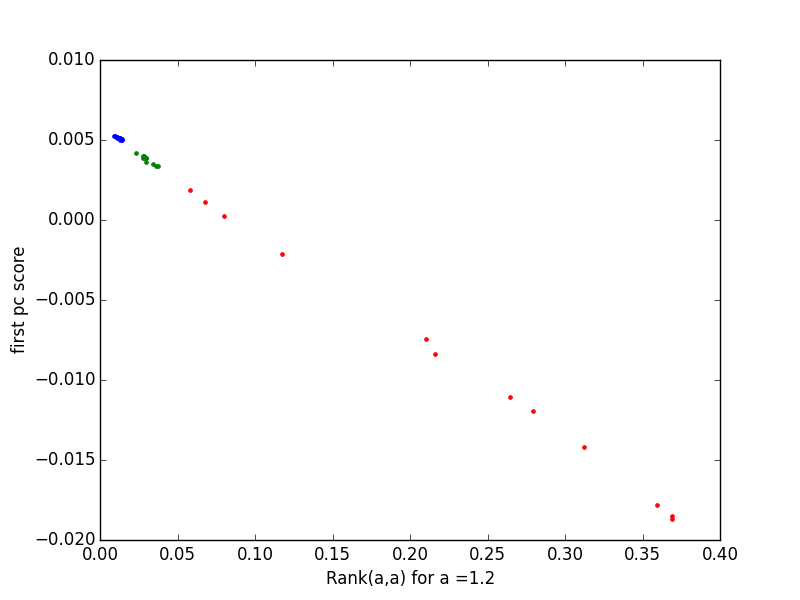}\\
(b)\includegraphics[width=0.4\textwidth]{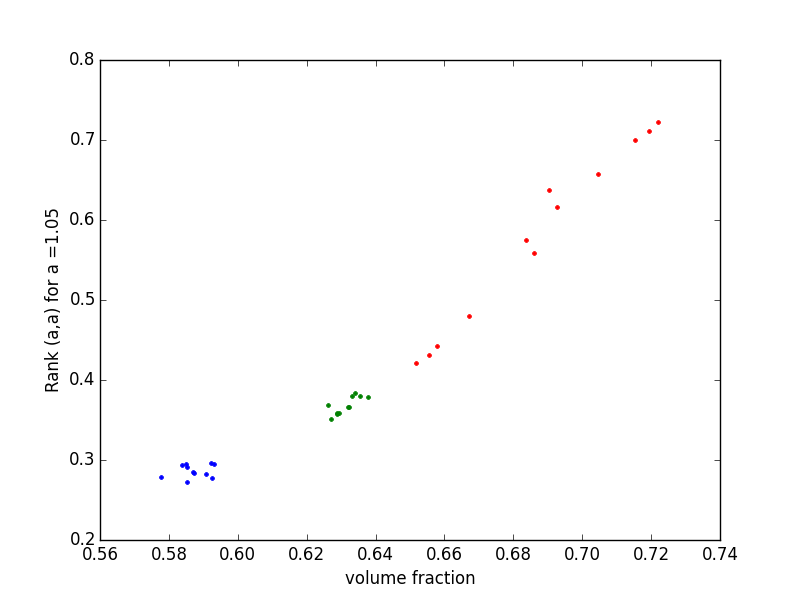}
	(e)\includegraphics[width=0.4\textwidth]{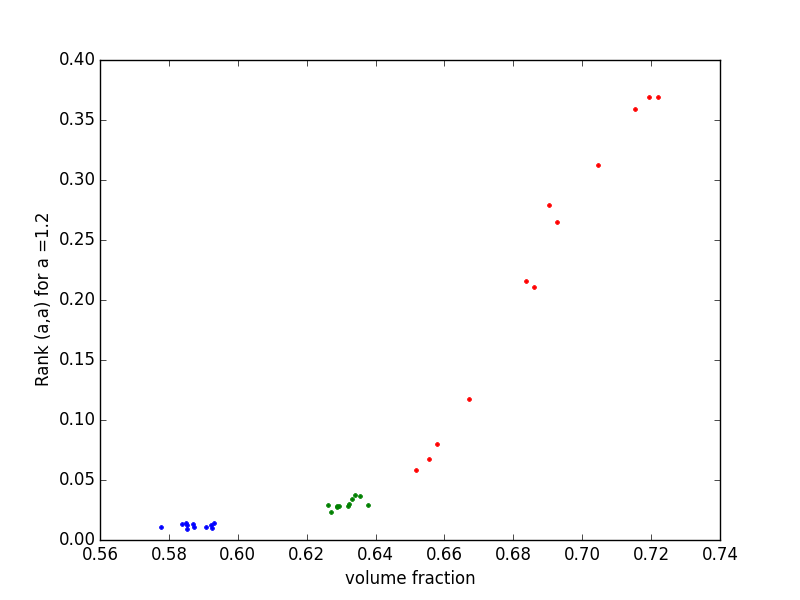} \\
(c)\includegraphics[width=0.4\textwidth]{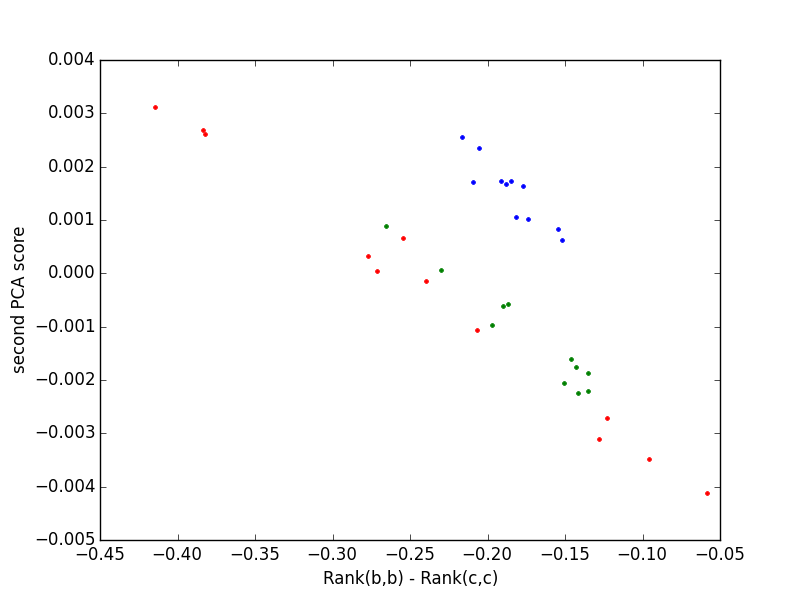}
	(f)\includegraphics[width=0.4\textwidth]{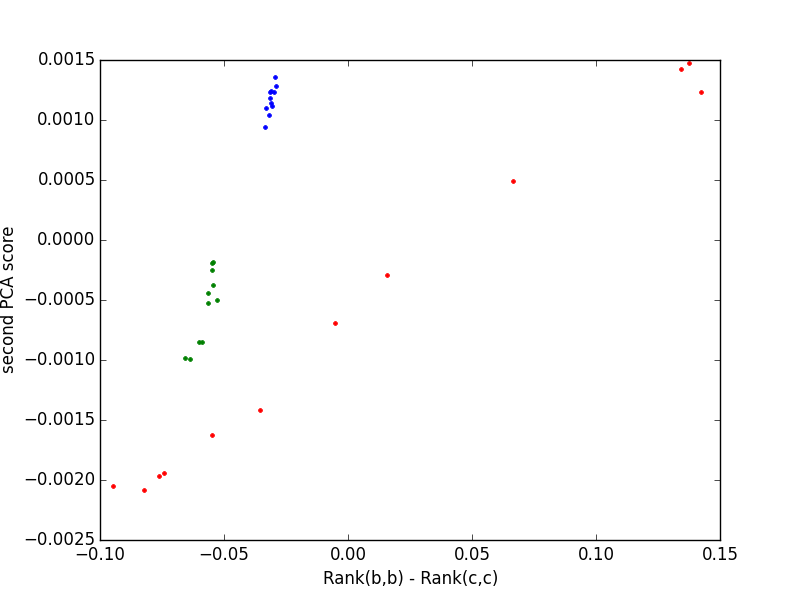}
\end{center}
\caption{ Geometric indices that correlate with first and second p.c. scores.   
(a) $R_1 = \beta_1(1.05,1.05)$ versus first p.c. scores of $\beta_1(s,t)$ for the 36 subsets. (b) Volume fraction versus $R_1$. (c) $D_1$ versus second p.c. score for $\beta_1(s,t)$.  
(d)  $R_2 = \beta_2(1.20,1.20)$ versus first p.c. scores of $\beta_2(s,t)$ for the 36 subsets. (e) Volume fraction versus $R_2$. (f) $D_2$ versus second p.c. score for $\beta_2(s,t)$.  
Blue dots are data for subsets from experiment $A$, green dots from $B$ and red dots are from the third, partially crystallised, packing $C$.}
\label{fig:beadpack_analysis}
\end{figure}

\clearpage


\section{Conclusions and future directions}

This paper has shown how the persistent homology rank function can be used as a summary statistic for distinguishing between spatial point processes generated by different models, and as the basis for a functional principal component analysis of point patterns arising in physical systems.    
Our results establish the basic techniques needed for further applications.  
For example, the work in Section~\ref{sec:CSR} could be extended to powerful topological summary statistics for spatial point patterns by using PCA or a machine learning approach for parameter estimation to determine which model a point pattern is likely to be drawn from.
In this paper we limited ourselves to \v Cech complexes when building the filtrations for persistence computations.  
It would be interesting to use filtrations derived from kernel density functions of point patterns as we believe this would be particularly powerful when looking at clustering models.

Our analysis of point patterns from colloids and sphere-packings can be extended to packings of non-spherical grains, or to any scalar function of interest via the connection that exists between Morse theory and persistent homology~\cite{edelsbrunner_hierarchical_2003}.  
As mentioned in Section~\ref{sec:background}, a general method for building a filtration of a space $X$ is via the sublevel sets of a continuous function $f: X \to \R$.  
To study the geometry of packings of non-spherical grains or any porous material, we can use the distance function defined by the shortest distance from each point in space to the surface of the grains.  
Other physical applications might study a function that specifies the local concentration or density of a particular phase, an energy potential, and so on.  
In this context a Gaussian random field plays the null hypothesis role analogous to a completely spatially random point process.  

Finally, there remain theoretical questions about analysis using persistent homology rank functions to explore. 
These include how robust the pairwise distances between the rank functions are under different choices of the weighting function $\phi$ (\ref{metric}) and whether there is  an optimal choice of $\phi$ under various scenarios. 
We also conjecture there are convergence properties for persistent homology rank functions constructed from point patterns over larger and larger domains after normalizing by volume. 
This should be particularly apparent where the point patterns in distant regions are independent. 
Complications arise from both boundary effects and long-range dependence, however.   
Even when the point patterns in disjoint regions are independent there can be homology classes that involve long-range behaviour.
This fact is one of the main obstacles to the statistical analysis of topological quantities, but also the reason for their sensitivity to physically-relevant non-local properties.


\bibliographystyle{unsrt}
\bibliography{PCAofPHRF}

\end{document}